\documentclass[reqno,12pt]{amsart}
\usepackage[margin=1in]{geometry}
\usepackage{amsfonts}
\usepackage[all]{xy}
\usepackage{enumitem}
\usepackage{amssymb,amsmath,url}
\usepackage{setspace}
\usepackage{amsthm}
\usepackage{cite}
\usepackage{color}
\numberwithin{equation}{section}
\usepackage[normalem]{ulem}
\usepackage{graphicx}
\usepackage{lmodern}
\graphicspath{ {images/} }
\usepackage{pdfpages}
\usepackage{mmacells}
\mmaDefineMathReplacement[≤]{<=}{\leq}
\mmaDefineMathReplacement[≥]{>=}{\geq}
\mmaDefineMathReplacement[≠]{!=}{\neq}
\mmaDefineMathReplacement[→]{->}{\to}[2]
\mmaDefineMathReplacement[⧴]{:>}{:\hspace{-.2em}\to}[2]
\mmaDefineMathReplacement{∉}{\notin}
\mmaDefineMathReplacement{∞}{\infty}
\mmaDefineMathReplacement{𝕕}{\mathbbm{d}}
\mmaSet{
	morefv={gobble=2},
	linklocaluri=mma/symbol/definition:#1,
	morecellgraphics={yoffset=1.9ex}
}

\author{Zhenghui Huo}
\author{Brett D. Wick}
\title{Compactness of operators on the Bergman space of the Thullen domain}
\begin{document}

\thanks{BDW's research is partially supported by National Science Foundation grants DMS \# 1560955 and DMS \# 11800057.}
	\address{Zhenghui Huo, Department of Mathematics and Statistics, The University of Toledo,  Toledo, OH 43606-3390, USA}
	\email{zhenghui.huo@utoledo.edu}
	\address{Brett D. Wick, Department of Mathematics and Statistics, Washington University in St. Louis,  St. Louis, MO 63130-4899, USA}
	\email{wick@math.wustl.edu}
		\newtheorem{thm}{Theorem}[section]
	\newtheorem{cl}[thm]{Claim}
	\newtheorem{lem}[thm]{Lemma}
		\newtheorem*{Rmk*}{Remark}
	\newtheorem{ex}[thm]{Example}
		\newtheorem{prop}[thm]{Proposition}
	\newtheorem{de}[thm]{Definition}
	\newtheorem{co}[thm]{Corollary}
	\newtheorem*{thm*}{Theorem}
		\newtheorem{qu}[thm]{Question}
	\maketitle
\begin{abstract}
We study compact operators on the Bergman space of the  Thullen domain defined by
$\{(z_1,z_2)\in \mathbb C^2: |z_1|^{2p}+|z_2|^2<1\}$ with $p>0$ and $p\neq 1$. The domain need not be smooth nor have a transitive automorphism group. We give a sufficient condition for the boundedness of various operators on the Bergman space. Under this boundedness condition, we characterize the compactness of operators on the Bergman space of the Thullen domain.

\medskip
\noindent
{\bf AMS Classification Numbers}:  32A07,  32A25, 32A36, 32A50, 47B35

\medskip

\end{abstract}

\section{Introduction}

  Let $\Omega$ be a domain in complex Euclidean space $\mathbb C^n$  and let $d\sigma$ be the Lebesgue measure. 
We use the symbol $\langle \cdot, \cdot\rangle$ and $\|\cdot\|$ to denote the inner product and the norm on $L^2(\Omega)$:
\begin{align}
&\langle f,g\rangle=\int_{\Omega}f(\zeta)\overline{g(\zeta)}d\sigma(\zeta),\\
&\|f\|=\sqrt{\langle f,f\rangle}.
\end{align}
The Bergman projection $P$ is the orthogonal projection from $L^2(\Omega)$ onto $A^2(\Omega)$, the closed subspace of square-integrable holomorphic functions on $\Omega$.
The kernel function associated to the projection $P$
 is called the Bergman kernel and is denoted by $K_{\Omega}$. For fixed $z\in \Omega$, we use $K_z$ to denote the function $K_\Omega(\cdot;\bar z)$ in $A^2(\Omega)$ and use $k_z$ to denote the normalized kernel function $K_z/\|K_z\|$. 
For a function $a\in L^{\infty}(\Omega)$, let $M_a:L^2(\Omega)\rightarrow L^2(\Omega)$ be the multiplication operator by $a$. Then the operator $T_a:=PM_a$ is called the Toeplitz operator with symbol $a$. Let $\mathcal L(A^2(\Omega))$ denote the space of bounded linear operators on $A^2(\Omega)$. The Toeplitz algebra $\mathcal T_{L^{\infty}}$ is the closed subalgebra of $\mathcal L(A^2(\Omega))$ generated by Toeplitz operators with $L^\infty$ symbols
 $$\mathcal T_{L^\infty}=\text{Closure}\left\{\sum_{k=1}^{K}\prod_{j=1}^{J}T_{u_{j,k}}:u_{j,k}\in L^\infty(\Omega) \text{ and } J,K\in \mathbb N \right\},$$
 where the closure is in the operator norm on $A^2(\Omega)$.

In a variety of classical function spaces, the  compactness of a given operator can be determined by examining only its behavior on $k_z$. A well-known result of Su\'arez \cite{Sua} showed that when $\Omega$ is the unit ball $\mathbb B^n$ in $\mathbb C^n$, an operator $T$ in $ \mathcal L(A^2(\mathbb B^n))$ is compact if and only if
 	$T$ is in $\mathcal T_{L^{\infty}}$
 	and $\lim_{\|z\|\rightarrow1}\|Tk_z\|=0$.  Su\'arez's results were later extended to various different function spaces and settings. Namely, the same results were shown to be true for the Bargmann-Fock space \cite{BAUER20121323}, Bergman spaces on the disc and unit ball with classical weights \cite{Mitkovski2013}, and the weighted Bergman spaces on the polydisc \cite{Mitkovski}. By introducing the notion of the Bergman-type space, a unified approach to many of these results was given in \cite{Mitkovski20142028}. Among these results, one of the key properties been used was that the domain $\Omega$ has a transitive automorphism group.  It is worth noting that using the $\bar{\partial}$-Neumann operator technique, versions of the Theorem 5.1 for $T$ in some subalgebra of $\tilde T_{L^\infty}$ have also been proved on more general domains in $\mathbb C^n$. See for example \cite{C̆uc̆ković2013,C̆uc̆ković2017}. 
 	
The Thullen domain we consider in this paper is defined by,
\begin{align}\label{1.3}
U^{\alpha}=\{z=(z_1,z_2)\in\mathbb C^2:|z_1|^{\frac{2}{\alpha}}+|z_2|^2<1\}, \text{ where }{\alpha>0, \alpha\neq 1}.
\end{align}
In 1931, Thullen \cite{Thullen} showed that the holomorphic automorphism $\varphi$ on $U^\alpha$ is of the form:
  \begin{align*}
  \varphi(z_1,z_2)=\left(e^{i\theta_1}z_1\left(\frac{\sqrt{1-|w|^2}}{1-z_2\bar w}\right)^{\alpha},e^{i\theta_2}\frac{w-z_2}{1-z_2\bar w} \right),
 \end{align*}
 where $|w|<1$ and $\theta_1,\theta_2\in \mathbb R$. The same result was also obtained by Cartan \cite{Cartan} using the Lie group approach. It's easy to see from the formula that the holomorphic automorphism group on $U^{\alpha}$ is not transitive. Therefore the Bergman space $A^2(U^{\alpha})$ does not fall into the category of Bergman-type spaces in \cite{Mitkovski20142028}.   However, a suitable modification of the technique in \cite{Mitkovski20142028} will work to study compactness of operators on $A^2(U^{\alpha})$.

The results in this paper are twofold. 
We give a sufficient condition for the boundedness of an operator whose adjoint and itself are defined a priori only on the linear span of the normalized reproducing kernels $\{k_z\}$. See Theorem 3.1. As a consequence, sufficient conditions for the boundedness of Toeplitz operators and Hankel operators can be obtained. See Corollaries 3.2 and 3.3.
 Then under the same sufficient condition we characterize the compact operators.  One of the main results (Theorem 5.3) in this paper shows that if the sufficient condition for the boundedness result is satisfied, then
 \begin{align*} & T \text{ is compact } \Longleftrightarrow\|Tk_z\| \text{ tends to } 0 \text{ as the point } z \text{ approaches the boundary of } U^{\alpha}. 
 \end{align*}
As a consequence, we also have 
 \begin{align*}  \text{ If } T \text{ is a Toeplitz operator with}& \text{ a } L^\infty \text{ symbol, then } \\T \text{ is compact } \Longleftrightarrow\|Tk_z\| \to 0 \text{ as the point } &z \text{ approaches the boundary of } U^{\alpha}. 
\end{align*}
See Corollary 5.2.

 In Section 2, we recall the explicit formula for the Bergman kernel function $K_{U^{\alpha}}$ on $U^{\alpha}$, and define two families of automorphisms $\{\phi_{z_1}\}$ and $\{\varphi_{z_2}\}$ on $U^{\alpha}$. We collect some basic properties of these automorphisms and then give some key lemmas.  Using these automorphisms,  we state and prove the boundedness results in Section 3. In Section 4, we give a geometric  decomposition of $U^{\alpha}$. See Proposition 4.1. With such a decomposition, we further show that an operator can be approximated by a series of compact operators. See Proposition 4.4. We state and prove the compactness results  in Section 5. We give some remarks and possible directions to generalize our results in Section 6.
\section{Preliminaries}
From now on, we let $U^{\alpha}$ be the domain $\Omega$ and let $\langle \cdot,\cdot\rangle$ and $\|\cdot\|$ denote the corresponding $L^2$ inner product and $L^2$ norm respectively. We define the weighted measure $d\lambda(w)$ on $U^\alpha$ to be $\|K_{w}\|^2d\sigma(w)$. Let $\mathbb D$ denote the unit disc in $\mathbb C$.  {Given} {functions} of several variables $f$ and $g$, we use $f\lesssim g$ to denote that $f\leq Cg$ for a constant $C$. If $f\lesssim g$ and $g\lesssim f$, then we say $f$ is comparable to $g$ and write $f\simeq g$. 

The explicit formula for the Bergman kernel function on the Thullen domain was obtained by Bergman \cite{Bergman1} and D'Angelo \cite{1978D'A1,1994D'A2}. Recall that the Thullen domain
 \begin{align*}
U^{\alpha}=\{z=(z_1,z_2)\in\mathbb C^2:|z_1|^{\frac{2}{\alpha}}+|z_2|^2<1\}, \text{ where }{\alpha>0, \alpha\neq 1}.
\end{align*} For $(z_1,z_2), (w_1,w_2)\in U^{\alpha}$, the Bergman kernel $K_{U^{\alpha}}$ is given by
\begin{align}\label{2.1}
	K_{U^{\alpha}}\left((z_1,z_2);(\bar w_1,\bar w_2)\right)= \frac{({\alpha}+1)(1-z_2\bar w_2)^{\alpha}+(\alpha-1)z_1\bar w_1}{\pi^2(1-z_2\bar w_2)^{2-\alpha}\left((1-z_2\bar w_2)^\alpha-{z_1\bar w_1}\right)^3}.
\end{align}
On the diagonal of $U^{\alpha}\times U^{\alpha}$,
 \begin{align}
 K_{U^{\alpha}}\left(w_1,w_2;\bar w_1,\bar w_2\right)&=\frac{({\alpha}+1)(1-| w_2|^2)^{\alpha}+({\alpha}-1) |w_1|^2}{\pi^2(1-|w_2|^2)^{2-\alpha}\left((1-|w_2|^2)^\alpha-{|w_1|}^2\right)^3}.\\
 &= \frac{({\alpha}+1)+(\alpha-1)\frac{|w_1|^2}{(1-|w_2|^2)^\alpha}}{\pi^2(1-|w_2|^2)^{2+\alpha}\left(1-\frac{|w_1|^2}{(1-|w_2|^2)^\alpha}\right)^3}.
 \end{align}
 For $(w_1,w_2)\in U^{\alpha}$, we have $\frac{|w_1|^2}{(1-|w_2|^2)^\alpha}<1$. Therefore $({\alpha}+1)+(\alpha-1)\frac{|w_1|^2}{(1-|w_2|^2)^\alpha}\simeq 1$ and 
 \begin{align}\label{KD}
 K_{U^{\alpha}}\left(w_1,w_2;\bar w_1,\bar w_2\right)
 &\simeq \frac{1}{(1-|w_2|^2)^{2+\alpha}\left(1-\frac{|w_1|^2}{(1-|w_2|^2)^\alpha}\right)^3}.
 \end{align}
 Similarly, we can obtain an estimate for the absolute value of $K_{U^{\alpha}}$ off the diagonal: 
 \begin{align}\label{KD1}
 |K_{U^{\alpha}}\left(z_1,z_2;\bar w_1,\bar w_2\right)|
 &\simeq \frac{1}{|1-z_2\bar w_2|^{2+\alpha}\left|1-\frac{z_1\bar w_1}{(1-z_2 \bar w_2)^\alpha}\right|^3}.
 \end{align}
We will use these estimates to simplify the computations that involves $K_{U^\alpha}$.  
 
As mentioned in Section 1, the holomorphic automorphism group on $U^\alpha$ is not transitive.  Still, for each point $z$ in $U^{\alpha}$, an (in general not holomorphic) automorphism of $U^{\alpha}$ that  sends $z$  to the origin can be constructed  and used to estimate the Bergman kernel $K_{U^{\alpha}}(\cdot;\bar z)$.  Our construction of such a mapping are as follows. For $(z_1,z_2)\in U^{\alpha}$, we define two mappings $\varphi_{z_2}$ and $\phi_{z_1}$ on $U^{\alpha}$:
 \begin{align}
 \varphi_{z_2}(w_1,w_2)&=\left(w_1\left(\frac{\sqrt{(1-|z_2|^2)}}{1-\bar z_2 w_2}\right)^\alpha,\frac{z_2-w_2}{1-\bar z_2 w_2}\right),\\\label{phi}
  \phi_{z_1}(w_1,w_2)&=\left(\frac{z_1-w_1}{1-\bar z_1 w_1},w_2\left(\frac{\sqrt{(1-|z_1|^2)}}{1-\bar z_1 w_1}\right)\sqrt{\frac{(1-|w_1|^2)(1-\left|\frac{z_1-w_1}{1-\bar z_1 w_1}\right|^{\frac{2}{\alpha}})}{(1-|w_1|^{\frac{2}{\alpha}})(1-\left|\frac{z_1-w_1}{1-\bar z_1 w_1}\right|^{2})}}\right).
 \end{align}
 For simplicity, we let $\phi^{(1)}(w)$ and $\phi^{(2)}(w)$ denote the first and second coordinates of $\phi_{z_1}(w)$. Then $\phi^{(2)}(w)$ can also be expressed as
 \begin{align}
\phi^{(2)}(w)=w_2\left(\frac{\sqrt{(1-|z_1|^2)}}{1-\bar z_1 w_1}\right)\sqrt{\frac{(1-|w_1|^2)(1-\left|\phi^{(1)}(w)\right|^{\frac{2}{\alpha}})}{(1-|w_1|^{\frac{2}{\alpha}})(1-\left|\phi^{(1)}(w)\right|^{2})}}.
\end{align}
 Both $\varphi_{z_2}$ and $\phi_{z_1}$ are involutions on $U^{\alpha}$. $\varphi_{z_2}$ is a holomorphic automorphism while $\phi_{z_1}$ is not holomorphic unless $\alpha=1$. When $\alpha=1$, the domain $U^{1}$ is the unit ball $\mathbb B^2$. 
 
 By the biholomorphic transformation formula for the Bergman kernel (See \cite{Krantz}), 
 \begin{align}\label{trans}
 K_{U^{\alpha}}(w;\bar w)=K_{U^\alpha}(\varphi_{z_2}(w);\overline{\varphi_{z_2}(w)})|J\varphi_{z_2}(w)|^2.
 \end{align}
An analogue of formula (\ref{trans}) is not true for $\phi_{z_2}$ in general since $\phi_{z_2}$ is not biholomorphic. Instead, we have the following estimate:
 \begin{lem}
 	Let $z$ and $w$ be in $U^{\alpha}$ and let $\phi_{z_1}(w)$ be as in (\ref{phi}). Then
 	\begin{align}\label{1.13}
 	K_{U^{\alpha}}(w;\bar w)\simeq K_{U^{\alpha}}(\phi_{z_1}(w);\overline{\phi_{z_1}(w)})|J\phi_{z_1}(w)|^2\frac{(1-\left|\phi^{(2)}(w)\right|^{2})^{\alpha-1}}{{(1-|w_2|^2)^{\alpha-1}}}.
 	\end{align}
 \end{lem}
 \begin{proof}  Since $
 1-\left|\phi^{(1)}(w)\right|^2=\frac{(1-|z_1|^2)(1-|w_1|^2)}{|1-\bar z_1 w_1|^2}$, we have
 	\begin{align}\label{1.14}
 	\frac{|\phi^{(2)}(w)|^2}{(1-|\phi^{(1)}(w)|^{\frac{2}{\alpha}})}&=\frac{|w_2|^2{(1-|z_1|^2)}(1-|w_1|^2)(1-\left|\phi^{(1)}(w)\right|^{\frac{2}{\alpha}})}{|1-\bar z_1 w_1|^2(1-|w_1|^{\frac{2}{\alpha}})(1-\left|\phi^{(1)}(w)\right|^{2})(1-|\phi^{(1)}(w)|^{\frac{2}{\alpha}})}\nonumber\\
 	&=\frac{|w_2|^2}{(1-|w|^{\frac{2}{\alpha}})},
 	\end{align}
 	or equivalently, $
|\phi^{(2)}(w)|^2=|w_2|^2{(1-|\phi^{(1)}(w)|^{\frac{2}{\alpha}
}})(1-|w_1|^{\frac{2}{\alpha}})^{-1}.$

By (\ref{KD}) we have 
 \begin{align}
K_{U^{\alpha}}\left(w_1,w_2;\bar w_1,\bar w_2\right)\nonumber
&\simeq
 {(1-|w_2|^2)^{-2-\alpha}\left(1-\frac{|w_1|^2}{(1-|w_2|^2)^\alpha}\right)^{-3}}\nonumber\\
 &\simeq
 {(1-|w_2|^2)^{-2-\alpha}\left(1-\frac{|w_1|^{\frac{2}{\alpha}}}{(1-|w_2|^2)}\right)^{-3}}\nonumber\\
  &=
  {(1-|w_2|^2)^{1-\alpha}\left(1-|w_2|^2-{|w_1|^{\frac{2}{\alpha}}}\right)^{-3}}\nonumber\\
   &=
   {(1-|w_2|^2)^{1-\alpha}\left(1-\frac{|w_2|^2}{1-{|w_1|^{\frac{2}{\alpha}}}}\right)^{-3}(1-|w_1|^{\frac{2}{\alpha}})^{-3}}\nonumber\\
   &\simeq
   {(1-|w_2|^2)^{1-\alpha}\left(1-\frac{|w_2|^2}{1-{|w_1|^{\frac{2}{\alpha}}}}\right)^{-3}(1-|w_1|^2)^{-3}}.
 \end{align}
The second and the last approximation signs above hold by the fact that $\frac{1-r^2}{1-r^p}\simeq 1$ for any $r\in [0,1)$ and $p>0$. Similarly, we obtain 
  \begin{align}\label{1.16}
 K_{U^{\alpha}}\left(\phi_{z_1}(w);\overline{\phi_{z_1}(w)}\right)
 &\simeq
 {(1-|\phi^{(2)}(w)|^2)^{1-\alpha}\left(1-\frac{|\phi^{(2)}(w)|^2}{1-{|\phi^{(1)}(w)|^{\frac{2}{\alpha}}}}\right)^{-3}(1-|\phi^{(1)}(w)|^{2})^{-3}}.
 \end{align}
 Applying (\ref{1.14}) to the right hand side of (\ref{1.16}) then yields:
 \begin{align}\label{1.17}
 K_{U^{\alpha}}\left(\phi_{z_1}(w);\overline{\phi_{z_1}(w)}\right)&\simeq{(1-|\phi^{(2)}(w)|^2)^{1-\alpha}\left(1-\frac{|w_2|^2}{1-{|w_1|^{\frac{2}{\alpha}}}}\right)^{-3}(1-|\phi^{(1)}(w)|^{2})^{-3}}.
 \end{align}
 We turn to compute $J\phi_{z_1}(w)$.
 Since $\frac{\partial }{\partial w_2}\phi^{(1)}(w)=0$, we have 
 \begin{align}\label{1.18}
 J\phi_{z_1}(w)&=\frac{\partial }{\partial w_1}\phi^{(1)}(w)\frac{\partial }{\partial w_2}\phi^{(2)}(w)\nonumber\\
 &=\frac{(1-|z_1|^2)}{(1-\bar z_1 w_1)^2}\frac{\sqrt{(1-|z_1|^2)}}{1-\bar z_1 w_1}\sqrt{\frac{(1-|w_1|^2)(1-\left|\phi^{(1)}(w)\right|^{\frac{2}{\alpha}})}{(1-|w_1|^{\frac{2}{\alpha}})(1-\left|\phi^{(1)}(w)\right|^{2})}}\nonumber\\
&\simeq\left(\frac{1-|z_1|^2}{(1-\bar z_1 w_1)^2}\right)^{\frac{3}{2}}.
 \end{align}
 Combining (\ref{1.17}) and (\ref{1.18}) gives the desired estimate (\ref{1.13}):
 \begin{align}
 &|J\phi_{z_1}(w)|^2K_{U^{\alpha}}\left(\phi_{z_1}(w);\overline{\phi_{z_1}(w)}\right)\nonumber\\
 \simeq&\left(\frac{1-|z_1|^2}{|1-\bar z_1 w_1|^2}\right)^3{(1-|\phi^{(2)}(w)|^2)^{1-\alpha}\left(1-\frac{|w_2|^2}{1-{|w_1|^{\frac{2}{\alpha}}}}\right)^{-3}(1-|\phi^{(1)}(w)|^{2})^{-3}}
\nonumber\\
\simeq&{(1-|\phi^{(2)}(w)|^2)^{1-\alpha}\left(1-\frac{|w_2|^2}{1-{|w_1|^{\frac{2}{\alpha}}}}\right)^{-3}(1-|w_1|^{2})^{-3}}
\nonumber\\
\simeq&\frac{(1-|\phi^{(2)}(w)|^2)^{1-\alpha}}{(1-|w_2|^2)^{1-\alpha}}K_{U^\alpha}(w,\bar w).\nonumber
\end{align}
 \end{proof}

   We need a Forelli-Rudin type estimate on the domain $U^{\alpha}$. Such an estimate can be proved  using the following lemma. See for example \cite{Zhu}.
  	\begin{lem}Let $\theta$ denote Lebesgue measure on the unit sphere $\mathbb S^k\subset\mathbb C^k$.
  		For $\epsilon<1$ and $w\in\mathbb B^k$, let 
  		\begin{equation}\label{**}
  		a_{\epsilon,\delta}(w)=\int_{\mathbb B^k}\frac{(1-|\eta|^2)^{-\epsilon}}{|1-\langle w,\eta\rangle |^{1+k-\epsilon-\delta}}d\sigma(\eta),
  		\end{equation}
  		and let
  		\begin{equation}\label{ }
  		b_\delta(w)=\int_{\mathbb S^k}\frac{1}{|1-\langle w,\eta\rangle |^{k-\delta}}d\theta(\eta).
  		\end{equation}
  		Then \begin{enumerate}
  			\item for $\delta>0$, both $a_{\epsilon,\delta}$ and $b_{\delta}$ are bounded on $\mathbb B^k$.
  			\item for $\delta=0$, both $a_{\epsilon,\delta}(w)$ and $b_{\delta}(w)$ are comparable to the function $-\log(1-|w|^2)$.
  			\item for $\delta<0$, both $a_{\epsilon,\delta}(w)$ and $b_{\delta}(w)$ are comparable to the function $(1-|w|^2)^{\delta}$.
  		\end{enumerate}
  	\end{lem}
  	Here we state a Forelli-Rudin type estimate on $U^{\alpha}$:
  	\begin{lem}
  	For $\epsilon_1<\alpha+1, \epsilon_2<1$, $\epsilon_3>0$, and $z\in U^{\alpha}$, let
  	\begin{align}
  	I_{\delta_1,\delta_2}(z)=\int_{U^{\alpha}}\frac{(1-|w_2|^2)^{-\epsilon_2}(1-\frac{|w_1|^2}{(1-|w_2|^2)^{\alpha}})^{-\epsilon_1}(1-\frac{|z_1|^2}{(1-|z_2|^2)^{\alpha}})^{\epsilon_3}}{|1-\bar z_2 w_2|^{2+\alpha-\epsilon_2-\delta_2}\left|1-\frac{\bar z_1 w_1}{(1- |z_2|^2)^{\frac{\alpha}{2}}}\right|^{3-\epsilon_1+\epsilon_3-\delta_1}}d\sigma(w).
  	\end{align}
  	Then for $\delta_1\geq 0, \delta_2>0$, $I_{\delta_1,\delta_2}$ is bounded on $U^{\alpha}$.
  	\end{lem}
  	\begin{proof}
  		We first transform $I_{\delta_1,\delta_2}(z)$ into an integral on the polydisc $\mathbb D^2$:
  		\begin{align}\label{1.24}
  	I_{\delta_1,\delta_2}(z)=	&	\int_{U^{\alpha}}\frac{(1-|w_2|^2)^{-\epsilon_2}(1-\frac{|w_1|^2}{(1-|w_2|^2)^{\alpha}})^{-\epsilon_1}(1-\frac{|z_1|^2}{(1-|z_2|^2)^{\alpha}})^{\epsilon_3}}{|1-\bar z_2 w_2|^{2+\alpha-\epsilon_2-\delta_2}\left|1-\frac{\bar z_1 w_1}{(1-|z_2|^2)^{\frac{\alpha}{2}}}\right|^{3-\epsilon_1+\epsilon_3-\delta_1}}d\sigma(w)\nonumber\\
  	=&\int_{\mathbb D^2}\frac{(1-|w_2|^2)^{\alpha-\epsilon_2}(1-|t_1|^2)^{-\epsilon_1}(1-\frac{|z_1|^2}{(1-|z_2|^2)^{\alpha}})^{\epsilon_3}}{|1-\bar z_2 w_2|^{2+\alpha-\epsilon_2-\delta_2}\left|1-\frac{\bar z_1 t_1(1-|w_2|^2)^{\frac{\alpha}{2}}}{(1-|z_2|^2)^{\frac{\alpha}{2}}}\right|^{3-\epsilon_1+\epsilon_3-\delta_1}}d\sigma(t_1,w_2)
  	\nonumber\\
  	=&\int_{\mathbb D}\frac{(1-|w_2|^2)^{\alpha-\epsilon_2}}{|1-\bar z_2 w_2|^{2+\alpha-\epsilon_2-\delta_2}}\int_{\mathbb D}\frac{(1-|t_1|^2)^{-\epsilon_1}(1-\frac{|z_1|^2}{(1-|z_2|^2)^{\alpha}})^{\epsilon_3}}{\left|1-\frac{\bar z_1 t_1(1-|w_2|^2)^{\frac{\alpha}{2}}}{(1-|z_2|^2)^{\frac{\alpha}{2}}}\right|^{3-\epsilon_1+\epsilon_3-\delta_1}}d\sigma(t_1)d\sigma(w_2).
  		\end{align}
  	We consider two cases: $\delta_1\geq1$, and $1>\delta_1\geq0$.
  	
  		If $\delta_1\geq 1$, then Lemma 2.2 implies that 
  		\begin{align}
  		\label{1.25}
  \int_{\mathbb D}\frac{(1-|t_1|^2)^{-\epsilon_1}(1-\frac{|z_1|^2}{(1-|z_2|^2)^{\alpha}})^{\epsilon_3}}{\left|1-\frac{\bar z_1 t_1(1-|w_2|^2)^{\frac{\alpha}{2}}}{(1-|z_2|^2)^{\frac{\alpha}{2}}}\right|^{3-\epsilon_1+\epsilon_3-\delta_1}}d\sigma(t_1)<C_1\end{align}
  for some constant $C_1$.	Then substituting inequality (\ref{1.25}) into (\ref{1.24}) and applying Lemma 2.2 again yield 
  $$I_{\delta_1,\delta_2}(z)<C_1\int_{\mathbb D}\frac{(1-|w_2|^2)^{\alpha-\epsilon_2}}{|1-\bar z_2 w_2|^{2+\alpha-\epsilon_2-\delta_2}}d\sigma(w_2)<C_1C_2,$$
  for some constant $C_2$.
  
  If $1>\delta_1\geq0$, then Lemma 2.2 implies that
  \begin{align}
  \label{1.26}
  \int_{\mathbb D}\frac{(1-|t_1|^2)^{-\epsilon_1}(1-\frac{|z_1|^2}{(1-|z_2|^2)^{\alpha}})^{\epsilon_3}}{\left|1-\frac{\bar z_1 t_1(1-|w_2|^2)^{\alpha/2}}{(1-|z_2|^2)^{\frac{\alpha}{2}}}\right|^{3-\epsilon_1+\epsilon_3-\delta_1}}d\sigma(t_1)\simeq\left(1-\frac{|z_1|^2(1-|w_2|^2)^\alpha}{(1-|z_2|^2 )^{\alpha}}\right)^{\delta_1-1}.\end{align}
  Substituting (\ref{1.26}) into (\ref{1.24}) yields:
  \begin{align}\label{1.27}
 	I_{\delta_1,\delta_2}(z)
 	\simeq&\int_{\mathbb D}\frac{(1-|w_2|^2)^{\alpha-\epsilon_2}}{|1-\bar z_2 w_2|^{2+\alpha-\epsilon_2-\delta_2}}\left(1-\frac{|z_1|^2(1-|w_2|^2)^\alpha}{(1-|z_2|^2 )^{\alpha}}\right)^{\delta_1-1}d\sigma(w_2).
  \end{align}
  Using polar coordinates $t_2=r\eta$ for $r\in [0,1)$ and $\eta\in \mathbb S^1$, we have 
   \begin{align}\label{1.29}
   I_{\delta_1,\delta_2}(z)
   \simeq&\int_{0}^{1}\int_{\mathbb S^1}\frac{r(1-r^2)^{\alpha-\epsilon_2}}{|1-\bar z_2 r\eta|^{2+\alpha-\epsilon_2-\delta_2}}d\theta(\eta)\left(1-\frac{|z_1|^2(1-r^2)^\alpha}{(1-|z_2|^2)^{\alpha}}\right)^{\delta_1-1}dr.
   \end{align}
   Applying Lemma 2.2 and the substitution $s=r^2$ to the inner integral gives
    \begin{align}\label{1.30}
    I_{\delta_1,\delta_2}(z)
    \simeq&\int_{0}^{1}\frac{(1-s)^{\alpha-\epsilon_2}}{(1- |z_2|^2s )^{1+\alpha-\epsilon_2-\delta_2}}\left(1-\frac{|z_1|^2(1-s)^\alpha}{(1-|z_2|^2)^{\alpha}}\right)^{\delta_1-1}ds\nonumber\\
     \simeq&\int_{0}^{1}\frac{(1-s)^{\alpha-\epsilon_2}}{(1- |z_2|^2s )^{1+\alpha-\epsilon_2-\delta_2}}\left(1-\frac{|z_1|^{\frac{2}{\alpha}}(1-s)}{1-|z_2|^2}\right)^{\delta_1-1}ds
     \nonumber\\
     \lesssim&\int_{0}^{1}\frac{(1-s)^{\alpha-\epsilon_2}}{(1- |z_2|^2s )^{1+\alpha-\epsilon_2-\delta_2}}s^{\delta_1-1}ds
       \nonumber\\
       =&\int_{0}^{1/2}\frac{(1-s)^{\alpha-\epsilon_2}}{(1- |z_2|^2s )^{1+\alpha-\epsilon_2-\delta_2}}s^{\delta_1-1}ds+\int_{1/2}^{1}\frac{(1-s)^{\alpha-\epsilon_2}}{(1- |z_2|^2s )^{1+\alpha-\epsilon_2-\delta_2}}s^{\delta_1-1}ds.
    \end{align}
Since 
$\int_{0}^{1/2}\frac{(1-s)^{\alpha-\epsilon_2}}{(1- |z_2|^2s )^{1+\alpha-\epsilon_2-\delta_2}}s^{\delta_1-1}ds\lesssim \int_{0}^{1/2}s^{\delta_1-1}ds\lesssim 1$,
it suffices to show that 
\begin{align}\int_{1/2}^{1}\frac{(1-s)^{\alpha-\epsilon_2}}{(1- |z_2|^2s )^{1+\alpha-\epsilon_2-\delta_2}}s^{\delta_1-1}ds\lesssim 1.\end{align}
When $|z_2|\in [0,1/2]$, $(1-|z_2|^2s)\simeq 1$. We have
   \begin{align}
&\int_{1/2}^{1}\frac{(1-s)^{\alpha-\epsilon_2}}{(1- |z_2|^2s )^{1+\alpha-\epsilon_2-\delta_2}}s^{\delta_1-1}ds\lesssim\int_{1/2}^{1}{(1-s)^{\alpha-\epsilon_2}}ds\lesssim 1,
  \end{align}
  When $|z_2|\in [1/2,1)$, we have
   \begin{align}
   &\int_{1/2}^{1}\frac{(1-s)^{\alpha-\epsilon_2}}{(1- |z_2|^2s )^{1+\alpha-\epsilon_2-\delta_2}}s^{\delta_1-1}ds\nonumber\\\lesssim&\int_{1/2}^{1}\frac{(1-s)^{\alpha-\epsilon_2}}{(1- |z_2|^2s )^{1+\alpha-\epsilon_2-\delta_2}}ds\nonumber\\
   \lesssim&\frac{-(1-s)^{1+\alpha-\epsilon_2}}{(1- |z_2|^2s )^{1+\alpha-\epsilon_2-\delta_2}}\bigg|_{\frac{1}{2}}^1+\int_{1/2}^{1}\frac{|z_2|^2(1-s)^{1+\alpha-\epsilon_2}}{(1- |z_2|^2s )^{2+\alpha-\epsilon_2-\delta_2}}ds\nonumber\\
  \lesssim &  1+\int_{1/2}^{1}\frac{|z_2|^2}{(1- |z_2|^2s )^{1-\delta_2}}ds
=1+{\delta_1^{-1}(1- |z_2|^2s )^{\delta_2}}\big|_{\frac{1}{2}}^1\lesssim 1.   \end{align}
 Therefore $ I_{\delta_1,\delta_2}(z)\lesssim 1$ on $U^{\alpha}$.
  	\end{proof}
  	To obtain the boundedness results for the operators on $A^2(U^{\alpha})$, we also need  Schur's lemma. See \cite{Zhu} for a proof. 
 	 	 \begin{lem}[Schur's Lemma] Let $(X,\mu)$ and $(X,\nu)$ be measure spaces, $R(x,y)$ a non-negative measurable function on $X\times X$, $1<p<\infty$ and $\frac{1}{p}+\frac{1}{q}=1$. Suppose $h$ is a positive function on $X$ that is measurable with respect to $\mu$ and $\nu$ and $C_p$ and $C_q$ are positive constants such that
 	 	 	\begin{align}\label{***}
 	 	 	&\int_{X}R(x,y)h(y)^qd\nu(y)\leq C_q h(x)^q \text{ for $\mu$-almost every x };\\
 	 	 	&\int_{X}R(x,y)h(y)^pd\nu(x)\leq C_q h(y)^p \text{ for $\nu$-almost every y }.
 	 	 	\end{align}
 	 	 	Then $Tf(x)=\int_XR(x,y)f(y)d\nu(y)$ defines a bounded operator $T:L^p(X;\nu)\mapsto L^p(X;\mu)$ with $\|T\|_{L^p(X;\nu)\mapsto L^p(X;\mu)}\leq C_q^{\frac{1}{q}}C_p^{\frac{1}{p}}$.
 	 	 \end{lem}
 \section{A sufficient condition for the boundedness}
 In this section we give and prove a sufficient condition for the boundedness of various operators on the the Bergman space of the Thullen domain. See (\ref{1.36}) and (\ref{1.37}). These two inequalities are stronger conditions for $L^2$ boundedness. In fact, they imply the $L^p$ boundedness for a range of $p$. See the Remark after the proof of Theorem 3.1. As one will see soon, when the operator $T$ is a Toeplitz operator with bounded symbol (Corollary 3.2), a Hankel operator with bounded symbol (Corollary 3.3),  $T$ satisfies this condition. 
 
 We begin by defining two translation operators on $L^2(U^\alpha)$ using $\varphi_{z_2}$ and $\phi_{z_1}$:
 \begin{align}
 &U_{z}f(w):=f(\varphi_{z_2}(w))J\varphi_{z_2}(w);\\
 &V_{z_1}f(w):=f(\phi_{z_1}(w))J\phi_{z_1}(w).\end{align}
Here  $J$ is the holomorphic Jacobian determinant. Since $\varphi_{z_2}$ is a biholomorphism on $U^{\alpha}$, the induced $U_z$ is an isometry on $L^2(U^{\alpha})$. Since $\phi^{(1)}(w)$ is a holomorphic function and $\phi^{(2)}(w)$ is holomorphic in $w_2$, we have
 $$d\sigma(\phi_{z_1}(w))=|J\phi_{z_1}(w)|^2d\sigma(w).$$
 Therefore the induced operator $V_{z_1}$ is also an isometry on $L^2(U^{\alpha})$. 
 	 \begin{thm}
 	Let $T:A^2(U^\alpha)\rightarrow A^2(U^{\alpha})$ be a linear operator defined on the linear span of the normalized reproducing kernels of $A^2(U^{\alpha})$. Assume that there exists an operator $T^*$ defined on the same span such that the duality relation $\langle T k_z,k_w\rangle =\langle k_z,T^*k_w\rangle$ holds for all $z,w \in U^{\alpha}$. 
 	If 
 	\begin{align}\label{1.36}
 	&\sup\limits_{z\in U^{\alpha}}\|V_{f_{\alpha}(z)}U_{z}Tk_{z}(w)\|_{L^p(U^{\alpha})}<\infty;\\\label{1.37}
 		&\sup\limits_{z\in U^{\alpha}}\|V_{f_{\alpha}(z)}U_{z}T^*k_{z}(w)\|_{L^p(U^{\alpha})}<\infty,
 	\end{align}
 	for $p>4$, then $T$ can be extended to a bounded operator on $A^2(U^{\alpha})$.
 \end{thm}
 \begin{proof} Since the linear span of all normalized reproducing kernels is dense in $A^2(U^{\alpha})$ it suffices to show that $\|Tf\|\leq f$ for all $f$ that are in the linear span of the reproducing kernels. Note that for any such $f$ we have 
 	\begin{align}
 	\|Tf\|^2&=\int_{U^{\alpha}}|\langle Tf,K_z\rangle|^2d\sigma(z)=\int_{U^{\alpha}}|\langle f,T^*K_z\rangle|^2d\sigma(z)\nonumber\\&\leq \int_{U^{\alpha}}\left|\int_{U^{\alpha}}\langle f,K_w\rangle\langle K_w,T^*K_z\rangle d\sigma(w)\right|^2d\sigma(z)
 	\\&\leq \int_{U^{\alpha}}\left|\int_{U^{\alpha}}\langle K_w,T^*K_z\rangle |f(w)|d\sigma(w)\right|^2d\sigma(z).
 	\end{align}
 	Set $Rf(z):=\int_{U^{\alpha}}|\langle K_w,T^*K_z\rangle |f(w)d\sigma(w)$. Then the $L^2$ regularity of $R$ will imply the $L^2$ regularity of $T$.
 	By Lemma 2.4, we need to prove that there exists an $\epsilon>0$ such that:
 	\begin{align}\label{1.41}
 &\int_{U^{\alpha}}\left| \langle T^*K_z,K_w\rangle \right| \left\| K_w \right\|^\epsilon d\sigma(w)\lesssim \|K_z\|^{\epsilon},\\
 \label{1.42}&\int_{U^{\alpha}}\left| \langle TK_z,K_w\rangle \right| \left\| K_w \right\|^{\epsilon} d\sigma(w)\lesssim \|K_z\|^{\epsilon}.
 	\end{align}
 	Here we give the proof for inequality (\ref{1.41}). Inequality (\ref{1.42}) follows by the same argument.  Recall that $d\lambda(w)=\|K_w\|^2d\sigma(w)$. Then
\begin{align}\label{1.43}
&\int_{U^{\alpha}}\left| \langle T^*K_z,K_w\rangle \right| \left\| K_w \right\|^\epsilon d\sigma(w)
=\left\| K_z\right\| \int_{U^{\alpha}} \left| \langle T^*k_z,k_w\rangle \right| \left\| K_w \right\|^{\epsilon-1}d\lambda(w).
\end{align}
Substituting $w=\varphi_{z_2}(t)$ into (\ref{1.43}) yields
\begin{align}\label{1.44}
\left\| K_z\right\| \int_{U^{\alpha}} \left| \langle T^*k_z,k_{\varphi_{z_2}(t)}\rangle \right| \left\| K_{\varphi_{z_2}(t)} \right\|^{\epsilon-1}d\lambda(\varphi_{z_2}(t)).
\end{align}
By (\ref{trans}), we have $d\lambda(\varphi_{z_2}(t))=d\lambda(t)$ and $\|K_{\varphi_{z_2}(t)}\||J\varphi_{z_2}(t)|=\|K_t\|$. Therefore 
\begin{align}\label{1.45}
&\left\| K_z\right\| \int_{U^{\alpha}} \left| \langle T^*k_z,k_{\varphi_{z_2}(t)}\rangle \right| \left\| K_{\varphi_{z_2}(t)} \right\|^{\epsilon-1}d\lambda(\varphi_{z_2}(t))\nonumber\\
=&\left\| K_z\right\| \int_{U^{\alpha}} \frac{\left|  T^*k_z(\varphi_{z_2}(t))\right|}{\|K_{\varphi_{z_2}(t)}\|} \left\| K_{\varphi_{z_2}(t)} \right\|^{\epsilon-1}d\lambda(t)\nonumber\\
=&\left\| K_z\right\| \int_{U^{\alpha}} \frac{\left|  T^*k_z(\varphi_{z_2}(t))J\varphi_{z_2}(t)\right|}{\|K_{t}\|} \left\| K_{\varphi_{z_2}(t)} \right\|^{\epsilon-1}d\lambda(t).
\end{align}
Recall $f_{\alpha}(z)=z_1/(1-|z_2|^2)^{\alpha/2}$. Substituting $t=\phi_{f_{\alpha}(z)}(w)$ into the integral above then gives
\begin{align}\label{1.46}
&\left\| K_z\right\| \int_{U^{\alpha}} \frac{\left|  T^*k_z(\varphi_{f_{\alpha}(z)}(\phi_{z_1}(w)))J\varphi_{z_2}(\phi_{z_1}(w))\right|}{\|K_{\phi_{f_{\alpha}(z)}(w)}\|} \left\| K_{\varphi_{z_2}(\phi_{z_1}(w))} \right\|^{\epsilon-1}d\lambda(\phi_{f_{\alpha}(z)}(w)).
\end{align}
By Lemma 2.1, we have 
\begin{align}
&d\lambda(\phi_{f_{\alpha}(z)}(w))=\frac{(1-|w_2|^2)^{\alpha-1}}{(1-|\phi^{(2)}(w)|^2)^{\alpha-1}}d\lambda(w);\\
&\|K_w\|=\|K_{\phi_{f_{\alpha}(z)}(w)}\||J\phi_{f_{\alpha}(z)}(w)|\frac{(1-|\phi^{(2)}(w)|^2)^{\frac{\alpha-1}{2}}}{(1-|w_2|^2)^{\frac{\alpha-1}{2}}}.
\end{align}
Thus (\ref{1.46}) becomes
\begin{align}\label{1.49}
&\left\| K_z\right\| \int_{U^{\alpha}} \frac{\left|  T^*k_z(\varphi_{z_2}(\phi_{f_{\alpha}(z)}(w)))J\varphi_{z_2}(\phi_{f_{\alpha}(z)}(w))J\phi_{z_1}(w)\right|(1-|w_2|^2)^{\frac{\alpha-1}{2}}\left\| K_{\varphi_{z_2}(\phi_{f_{\alpha}(z)}(w))} \right\|^{\epsilon-1}}{\|K_{w}\|(1-|\phi^{(2)}(w)|^2)^{\frac{\alpha-1}{2}}} d\lambda(w).
\end{align}
Recall that  $U_{z}f(w)=f(\varphi_{z_2}(w))J\varphi_{z_2}(w)$ and  $V_{z_1}f(w)=f(\phi_{z_1}(w))J\phi_{z_1}(w)$. Then (\ref{1.49}) can be expressed as follows:
\begin{align}\label{1.50}
&\left\| K_z\right\|^{\epsilon} \int_{U^{\alpha}} \frac{\left|  V_{f_\alpha(z)}U_{z}T^*k_z(w)\right|(1-|w_2|^2)^{\frac{\alpha-1}{2}}\left\| K_{\varphi_{z_2}(\phi_{f_{\alpha}(z)}(w))} \right\|^{\epsilon-1}\|K_z\|^{1-\epsilon}}{(1-|\phi^{(2)}(w)|^2)^{\frac{\alpha-1}{2}}\|K_{w}\|} d\lambda(w).
\end{align}
Applying H\"older's inequality to (\ref{1.50}) yields:
\begin{align}
&\left\| K_z\right\|^{\epsilon} \int_{U^{\alpha}} \frac{\left|  V_{f_\alpha(z)}U_{z}T^*k_z(w)\right|(1-|w_2|^2)^{\frac{\alpha-1}{2}}\left\| K_{\varphi_{z_2}(\phi_{f_{\alpha}(z)}(w))} \right\|^{\epsilon-1}\|K_z\|^{1-\epsilon}}{(1-|\phi^{(2)}(w)|^2)^{\frac{\alpha-1}{2}}\|K_{w}\|} d\lambda(w)\nonumber\\\label{1.51}
\leq&\left\| K_z\right\|^{\epsilon}\left( \int_{U^{\alpha}} \frac{(1-|w_2|^2)^{\frac{q(\alpha-1)}{2}}\left\| K_{\varphi_{z_2}(\phi_{f_{\alpha}(z)}(w))} \right\|^{q(\epsilon-1)}\|K_z\|^{q(1-\epsilon)}\|K_w\|^q}{(1-|\phi^{(2)}(w)|^2)^{\frac{q(\alpha-1)}{2}}} d\sigma(w)\right)^{\frac{1}{q}}\\
&\times\left(\int_{U^{\alpha}} |V_{f_{\alpha}(z)}U_{z}T^*k_z(w)|^pd\sigma(w)\right)^\frac{1}{p}.
\end{align}
We claim by choosing appropriate $p$ and $\epsilon$, the integral in (\ref{1.51}) is bounded as a function of $z$ on $ U^{\alpha}$. Substituting $t=\phi_{f_\alpha(z)}(w)$ into (\ref{1.51}) gives
\begin{align}
 &\int_{U^{\alpha}} \frac{(1-|w_2|^2)^{\frac{q(\alpha-1)}{2}}\left\| K_{\varphi_{z_2}(\phi_{f_{\alpha}(z)}(w))} \right\|^{q(\epsilon-1)}\|K_z\|^{q(1-\epsilon)}\|K_w\|^q}{(1-|\phi^{(2)}(w)|^2)^{\frac{q(\alpha-1)}{2}}} d\sigma(w)\nonumber\\
=&\int_{U^{\alpha}} \frac{(1-|\phi^{(2)}(t)|^2)^{\frac{q(\alpha-1)}{2}}\left\| K_{\varphi_{z_2}(t)} \right\|^{q(\epsilon-1)}\|K_z\|^{q(1-\epsilon)}\|K_{\phi_{f_{\alpha}(z)}(t)}\|^q|J\phi_{f_{\alpha}(z)}(t)|^2}{(1-|t|^2)^{\frac{q(\alpha-1)}{2}}}d\sigma(t)\nonumber\\
=&\int_{U^{\alpha}} \left\| K_{t} \right\|^{q(\epsilon-1)}|J\varphi_{z_2}(t)|^{q(1-\epsilon)}\|K_z\|^{q(1-\epsilon)}\|K_{t}\|^q|J\phi_{f_{\alpha}(z)}(t)|^{2-q}d\sigma(t)\nonumber\\\simeq&
\int_{U^{\alpha}}\frac{(1-|t_2|^2)^{\frac{-q\epsilon(2+\alpha)}{2}}(1-\frac{|t_1|^2}{(1-|t_2|^2)^{\alpha}})^{\frac{-3q\epsilon}{2}}(1-\frac{|z_1|^2}{(1-|z_2|^2)^{\alpha}})^{\frac{3(2-q)}{2}+\frac{3q(\epsilon-1)}{2}}}{|1-\bar z_2 t_2|^{(2+\alpha)q(1-\epsilon)}\left|1-\frac{\bar z_1 t_1}{(1-|z_2|^2)^{\frac{\alpha}{2}}}\right|^{3(2-q)}}d\sigma(t).
\end{align}
By Lemma 2.3, the integral above is bounded for $z\in U^{\alpha}$ if the following inequalities hold:
\begin{align}\label{1.54}
&\frac{3q\epsilon}{2}<1;\\\label{1.55}
&\frac{q\epsilon(2+\alpha)}{2}<1+\alpha;\\\label{1.56}
&\frac{3(2-q)}{2}+\frac{3q(\epsilon-1)}{2}>0;\\\label{1.57}
&(2+\alpha)q(1-\epsilon)+\frac{q\epsilon(2+\alpha)}{2}<2+\alpha;\\\label{1.58}
&3(2-q)+\frac{3q\epsilon}{2}-\frac{3(2-q)}{2}-\frac{3q(\epsilon-1)}{2}\leq 3.
\end{align}
The last inequality is trivial. Since $\alpha>0$, inequality (\ref{1.54}) implies (\ref{1.55}). Both (\ref{1.56}) and (\ref{1.57}) are equivalent to the $\epsilon>2-\frac{2}{q}$. Thus we have
$
2-\frac{2}{q}<\epsilon<\frac{2}{3q}.
$
Note that $2-\frac{2}{q}<\frac{2}{3q}$ when $q<\frac{4}{3}$. Hence for $q<\frac{4}{3}$, or equivalently for $p>4$, an $\epsilon$ can be chosen from $(2-\frac{2}{q},\frac{2}{3q})$. Therefore for $z\in U^{\alpha}$ the following integral is bounded:
\begin{align}
\int_{U^{\alpha}} \frac{(1-|w_2|^2)^{\frac{q(\alpha-1)}{2}}\left\| K_{\varphi_{z_2}(\phi_{f_{\alpha}(z)}(w))} \right\|^{q(\epsilon-1)}\|K_z\|^{q(1-\epsilon)}\|K_w\|^q}{(1-|\phi^{(2)}(w)|^2)^{\frac{q(\alpha-1)}{2}}} d\sigma(w)\lesssim 1.
\end{align}
For such  $p$ and $\epsilon$, we have 
\begin{align}\label{3.24}
\int_{U^{\alpha}}|\langle T^*K_z,K_w\rangle|\|K_w\|^\epsilon d\sigma(w)\lesssim \sup\limits_{z\in U^{\alpha}}\|V_{f_{\alpha}(z)}U_{z}T^*k_{z}(w)\|_{L^p(U^{\alpha})}\|K_z\|^{\epsilon}.
\end{align}
Similarly we obtain
\begin{align}\label{3.25}
\int_{U^{\alpha}}|\langle TK_z,K_w\rangle|\|K_w\|^\epsilon d\sigma(w)\lesssim \sup\limits_{z\in U^{\alpha}}\|V_{f_{\alpha}(z)}U_{z}Tk_{z}(w)\|_{L^p(U^{\alpha})}\|K_z\|^{\epsilon}.
\end{align}
Lemma 2.4 then implies that $T$ can be extended to a bounded operator on $A^2(U^{\alpha})$ if
\begin{align}
&\sup\limits_{z\in U^{\alpha}}\|V_{f_{\alpha}(z)}U_{z}Tk_{z}(w)\|_{L^p(U^{\alpha})}<\infty,\\
&\sup\limits_{z\in U^{\alpha}}\|V_{f_{\alpha}(z)}U_{z}T^*k_{z}(w)\|_{L^p(U^{\alpha})}<\infty.
\end{align}
\end{proof}
\paragraph{{\bf Remark}} It is worth noting that  when $p>4$, the inequalities (\ref{3.24}) and (\ref{3.25}) hold for all $\epsilon\in (\frac{2}{p},\frac{2}{3}-\frac{2}{3p})$. Using a variant of Schur's lemma in \cite{EM}, one can extend $T$ to a bounded operator on $A^{p^\prime}(U^{\alpha})$ for $p^\prime\in (\frac{p+2}{p-1},\frac{p+2}{3})$. Here we focus only on the $L^2$-boundedness of $T$. 
\vskip 10pt
In the case when $T$ in the above theorem is a Toeplitz operator or a Hankel operator, the conditions (\ref{1.36}) and (\ref{1.37}) have simpler forms. For the Toeplitz operator, we have the following corollary:
\begin{co}
Let $T_u$ be a Toeplitz operator whose symbol $u$ satisfies
\begin{align}
\sup_{z\in U^{\alpha}}\|u(\varphi_{z_2}(\phi_{f_{\alpha}(z)}(w))\|_{L^p(U^{\alpha})}<\infty,
\end{align}
	for $p>4$. Then $T_u$ is $L^2$-bounded.
\end{co}
\begin{proof} Recall the Bergman projection $P$ on $U^{\alpha}$.
Notice first that for $g\in L^2(U^{\alpha})$,
\begin{align}
T_ug(z)=P(ug)(z)=\int_{U^{\alpha}}u(w)g(w)\langle K_w,K_z\rangle d\sigma(w).
\end{align}
Therefore, it is enough to show that
\begin{align}\label{1.64}
\int_{U^{\alpha}}|u(w)||\langle K_w,K_z\rangle| \|K_w\|^{\epsilon}d\sigma(w)\lesssim \|K_z\|^{\epsilon}.
\end{align}
By the same argument as in the proof of Theorem 3.1, the integral on the left hand side above was controlled from above by
\begin{align}
\left\| K_z\right\|^{\epsilon} \left(\int_{U^{\alpha}} \left|  V_{f_\alpha(z)}U_{z}k_z(w)u(\varphi_{z_2}(\phi_{f_{\alpha}(z)}(w))\right|^pd\lambda(w)\right)^{\frac{1}{p}},
\end{align}
for $p>4$. We claim $|V_{f_{\alpha}(z)}U_{z}k_z(w)|\lesssim 1$. Then 
\begin{align}
\int_{U^{\alpha}}|u(w)\langle K_w,K_z\rangle| \|K_w\|^{\epsilon}d\sigma(w)\lesssim \|K_z\|^{\epsilon}\sup_{z\in U^{\alpha}}\|u(\varphi_{z_2}(\phi_{f_{\alpha}(z)}(w))\|_{L^p(U^{\alpha})},
\end{align}
and by Lemma 2.4 the proof is complete. The biholomorphic transformation formula gives 
\begin{align}
U_{z}k_z(w)&=K(\varphi_{z_2}(w);\bar z)\|K_z\|^{-1}J\varphi_{z_2}(w)\nonumber\\
&=\frac{K(w;\overline {\varphi_{z_2}(z)})\overline {J\varphi_{z_2}(z)}}{\|K_z\|}\nonumber\\&=\frac{K\left(w;\frac{\bar z_1}{(1-|z_2|^2)^{\frac{\alpha}{2}}},0\right)\overline {J\varphi_{z_2}(z)}}{\|K_{\varphi_{z_2}(z)}\||J\varphi_{z_2}(z)|}
\nonumber\\&=\frac{\left(1-\frac{|z_1|^2}{(1-|z_2|^2)^{\alpha}}\right)^{\frac{3}{2}}\overline {J\varphi_{z_2}(z)}}{\left(1-\frac{\bar z_1 w_1}{(1-|z_2|^2)^{\frac{\alpha}{2}}}\right)^3|J\varphi_{z_2}(z)|}.
\end{align}
Let $\phi^{(1)}(w)$ denote the first coordinate of $\phi_{f_{\alpha}(z)}(w)$. Note that $\phi^{(1)}(w)$ is the M\"obius map of $w_1$ that sends the origin to $f_{\alpha}(z)$. Hence 
\begin{align}
1-\frac{\bar z_1\phi^{(1)}(w)}{(1-|z_2|^2)^{\frac{\alpha}{2}}}=\frac{1-\frac{|z_1|^2}{(1-|z_2|^2)^{\alpha}}}{1-\frac{\bar z_1 w_1}{(1-|z_2|^2)^{\frac{\alpha}{2}}}}.
\end{align}
Therefore 
\begin{align}\label{1.69}
|V_{f_{\alpha}(z)}U_{z}k_z(w)|&=\frac{\left(1-\frac{|z_1|^2}{(1-|z_2|^2)^{\alpha}}\right)^{\frac{3}{2}}}{\left|1-\frac{\bar z_1 \phi^{(1)}(w)}{(1-|z_2|^2)^{\frac{\alpha}{2}}}\right|^3}|J\phi_{f_\alpha(z)}(w)|\nonumber\\&\simeq\frac{\left|1-\frac{\bar z_1 w_1}{(1-|z_2|^2)^{\frac{\alpha}{2}}}\right|^3}{\left(1-\frac{|z_1|^2}{(1-|z_2|^2)^{{\alpha}}}\right)^\frac{3}{2}}\frac{\left(1-\frac{|z_1|^2}{(1-|z_2|^2)^{{\alpha}}}\right)^\frac{3}{2}}{\left|1-\frac{\bar z_1 w_1}{(1-|z_2|^2)^{\frac{\alpha}{2}}}\right|^3}=1.
\end{align}
\end{proof}
Similarly, we treat the case of Hankel operators. Hankel operator $H_u:A^2(U^{\alpha})\rightarrow L^2(U^{\alpha})$ with symbol $u$ is defined by $H_uf=(I-P)M_uf$, where $P$ is the Bergman projection. A similar argument gives us the following result about $H_a$.
\begin{co}
	Let $H_u$ be a Hankel operator whose symbol $u$ satisfies
\begin{align}
\sup_{z\in U^{\alpha}}\|u(z)-u(\varphi_{z_2}(\phi_{f_{\alpha}(z)}(w)))\|_{L^p(U^{\alpha})}<\infty,
\end{align}
		for $p>4$ then $H_u$ is $L^2$-bounded.
\end{co}
\section{A geometric decomposition of $U^{\alpha}$}
A geometric decomposition of $U^{\alpha}$ plays an important role in the proof the compactness theorem. 
Our decomposition result uses the Skwarczy\'nski distance \cite{Skwarczynski}. We recall its definition here.

Let $\Omega$ be a bounded domain in $\mathbb C^n$ and let $K_{\Omega}$ be the Bergman kernel on $\Omega$.  Then the Skwarczy\'nski distance  $d(\cdot,\cdot)$ on  $\Omega$ is defined by 
$$d(z,w)=(1-|\langle k_z,k_w\rangle|)^{\frac{1}{2}}=\left(1-\frac{|K_{\Omega}(z;\bar w)|}{K_{\Omega}(z;\bar z)^{\frac{1}{2}}K_{\Omega}(w;\bar w)^{\frac{1}{2}}}\right)^{\frac{1}{2}}.$$
By its definition $0\leq d(z,w)\leq 1$. On the Thullen domain $U^\alpha$,
the kernel function $K_{U^\alpha}$ does not vanish on $U^{\alpha}$, and $K_{U^{\alpha}}(z;\bar z)^{-1}=0$ only when $z\in \mathbf bU^{\alpha}$. Hence the the Skwarczy\'nski distance  $d(\cdot,\cdot)$ on $U^\alpha$ satisfies the following: for $w\in U^{\alpha}$, the distance $d(z,w)=1$ if and only if $z$ is a boundary point of $U^{\alpha}$. 

Our decomposition result for $U^\alpha$ are as follows:
\begin{prop}
	The metric space $(U^{\alpha},d)$ satisfies the following property. For $r$ that is sufficiently close to 1, there exists an integer $N(r)$ and a constant $C(r)$ such that  there is  covering $\mathcal F_r=\{F_j\}$ of $U^{\alpha}$ by disjoint Borel sets satisfying
	\begin{enumerate}
		\item every point of $\Omega$ belongs to at most $N(r)$ of sets $G_j:=\{z\in U^{\alpha}:d(z,F_j)\leq r\}$.
		\item $\sup\operatorname{diam}_d F_j<C(r)$ for every $j$.
	\end{enumerate}
\end{prop}
\paragraph{\bf Remark} In general, if a metric space satisfies Proposition 4.1 and the constant $N(r)$ above does not depend on $r$, i.e. $N(r)\lesssim 1$, then  the metric space is said to have finite asymptotic dimension in the sense of Gromov \cite{Gromov}. The finiteness of the asymptotic dimension is satisfied for nice domains equipped with the Bergman metric such as the unit ball \cite{Sua} and polydisc \cite{Mitkovski} in $\mathbb C^n$.  We are able to show that  Proposition 4.1 holds for the domain $U^{\alpha}$ equipped with Skwarczynski distance. However the finiteness of the asymptotic dimension for the metric space $(U^{\alpha},d)$ is unclear to us.
\vskip 5pt
Let $s(z,w)$ denote $|\langle k_z,k_w\rangle|$ and set $p(x):=\sqrt{1-x}$. Then the Skwarczy\'nski distance $d(z,w)=p(s(z,w))$. Let $D(z,r)$ denote the ball centered at point $z$ of radius $r$ under this metric. If the distance between $z$ and $w$ is fixed, then we simply use $s$ to denote $s(z,w)$ and use $D(z,p(s))$ to denote the ball $D(z,r)$ with radius $r=p(s)$. 

The next lemma below shows that for any point $z\in U^{\alpha}$ the image of the ball $D(z,p(s))$ under the mapping $\phi_{f_{\alpha}(z)}\circ\varphi_{z_2}$ has a size that is comparable to the size of the ball $D(0,p(cs^\beta))$ for some constants $c, \beta>0$. 
	\begin{lem}
Let $z$ and $w$ be two points in $U^{\alpha}$. If $-\log s(z,w)$ is sufficiently large, then $-\log s(z,w)\simeq -\log s(\phi_{f_{\alpha}(z)}\circ\varphi_{z_2}(w),0)$. Moreover, for sufficiently large $-\log s$, there exists constants $C_1, C_2, a, b$ (independent of $s$) and a constant $c(s)$ such that the weighted measure of the ball $D(z,p(s))$ satisfies the following two properties:  \begin{enumerate}\item $D(0,p(C_1s^a)\subseteq \phi_{f_{\alpha}(z)}\circ\varphi_{z_2} (D(z,p(s)))\subseteq D(0,p(C_2s^b))$,  and  \item $c(s)\lambda(D(0,p(C_1s^a)))\lesssim \lambda(D(z,p(s)))\lesssim c(s)\lambda(D(0,p(C_2s^b)))$.\end{enumerate}
	\end{lem}
	\begin{proof} 
		Let $z$ and $w$ be two points in $U^{\alpha}$ such that $-\log s(z,w)\gg |\log \alpha|$. Then $-\log s(z,w)$ is large when it compares to $\left|\log \frac{1-r^2}{1-r^{\frac{2}{\alpha}}}\right|$ for all $r\in [0,1)$. By the definition of $s(\cdot,\cdot)$,
		\begin{align}
		-\log s(z,w)\simeq&-\log\frac{|K_{U^{\alpha}}(z,\bar w)|}{K_{U^\alpha}(z;\bar z)^{\frac{1}{2}}K_{U^\alpha}(w;\bar w)^{\frac{1}{2}}}&\nonumber
		\\\simeq&\log\frac{\left|1-\frac{\bar z_1 w_1}{(1-\bar z_2 w_2)^{\alpha}}\right|^3|1-z_2\bar w_2|^{2+\alpha}}{\left(1-\frac{|z_1|^2}{(1-|z_2|^2)^{\alpha}}\right)^{\frac{3}{2}}(1-|z_2|^2)^{1+\frac{\alpha}{2}}\left(1-\frac{|w_1|^2}{(1-|w_2|^2)^{\alpha}}\right)^{\frac{3}{2}}(1-|w_2|^2)^{1+\frac{\alpha}{2}}}
			\nonumber\\
			\simeq&\log\frac{\left|1-\frac{\bar z_1 w_1}{(1-\bar z_2 w_2)^{\alpha}}\right|^2}{\left({1-\frac{|z_1|^2}{(1-|z_2|^2)^\alpha}}\right)}-\log\left({1-\frac{|w_1|^2}{(1-|w_2|^2)^\alpha}}\right)-\log\frac{(1-|z_2|^2)(1-|w_2|^2)}{|1-\bar z_2 w_2|^2}.
		\end{align}
		By the definition of $\phi_{f_\alpha}$ and $\varphi_{z_2}$, we have
			\begin{align}
			&\log\frac{\left|1-\frac{\bar z_1 w_1}{(1-\bar z_2 w_2)^{\alpha}}\right|^2}{\left({1-\frac{|z_1|^2}{(1-|z_2|^2)^\alpha}}\right)}-\log\left({1-\frac{|w_1|^2}{(1-|w_2|^2)^\alpha}}\right)-\log\frac{(1-|z_2|^2)(1-|w_2|^2)}{|1-\bar z_2 w_2|^2}\nonumber\\
			=&\log\frac{1-|w_1\frac{(1-|z_2|^2)^{\alpha/2}}{(1-\bar z_2 w_2)^{\alpha}}|^2}{1-{|\phi^{(1)}(\varphi_{z_2}(w))|^{2}}}-\log\left({1-\frac{|w_1\frac{(1-|z_2|^2)^{\alpha/2}}{(1-\bar z_2 w_2)^{\alpha}}|^2}{(1-|\frac{z_2-w_2}{1-\bar z_2 w_2}|^2)^\alpha}}\right)-\log\frac{(1-|z_2|^2)(1-|w_2|^2)}{|1-\bar z_2 w_2|^2}.
			\end{align}
		Since $d(z,w)\gg \left|\log \frac{1-r^2}{1-r^{\frac{2}{\alpha}}}\right|$
		for all $r\in [0,1)$, we have $d(z,w)\simeq d(z,w) \pm \left|\log \frac{1-r^2}{1-r^{\frac{2}{\alpha}}}\right|$. Therefore 
			\begin{align}
&\log\frac{1-|w_1\frac{(1-|z_2|^2)^{\alpha/2}}{(1-\bar z_2 w_2)^{\alpha}}|^2}{1-{|\phi^{(1)}(\varphi_{z_2}(w))|^{2}}}-\log\left({1-\frac{|w_1\frac{(1-|z_2|^2)^{\alpha/2}}{(1-\bar z_2 w_2)^{\alpha}}|^2}{(1-|\frac{z_2-w_2}{1-\bar z_2 w_2}|^2)^\alpha}}\right)-\log\frac{(1-|z_2|^2)(1-|w_2|^2)}{|1-\bar z_2 w_2|^2}\nonumber\\
\simeq&-\log\left({1-{|\phi^{(1)}(\varphi_{z_2}(w))|^{2}}}\right)-\log\left({1-\frac{|\frac{z_2-w_2}{1-\bar z_2 w_2}|^2}{1-{|w_1\frac{(1-|z_2|^2)^{\alpha/2}}{(1-\bar z_2 w_2)^{\alpha}}|^{\frac{2}{\alpha}}}}}\right)
\nonumber\\
\simeq&-\log\left({1-{|\phi^{(1)}(\varphi_{z_2}(w))|^{2}}}\right)-\log\left({1-\frac{|\phi^{(2)}(\varphi_{z_2}(w))|^2}{1-{|\phi^{(1)}(\varphi_{z_2}(w))|^{\frac{2}{\alpha}}}}}\right)
	\nonumber\\
	\simeq&-\log\left({1-{|\phi^{(1)}(\varphi_{z_2}(w))|^{\frac{2}{\alpha}}}{-|\phi^{(2)}(\varphi_{z_2}(w))|^2}}\right)
	\nonumber\\
	\simeq&-\log\left({1-\frac{|\phi^{(1)}(\varphi_{z_2}(w))|^{\frac{2}{\alpha}}}{1-|\phi^{(2)}(\varphi_{z_2}(w))|^2}}\right)-\log\left({1-|\phi^{(2)}(\varphi_{z_2}(w))|^2}\right)	
	\nonumber\\
	\simeq&-\log\left({1-\frac{|\phi^{(1)}(\varphi_{z_2}(w))|^2}{(1-|\phi^{(2)}(\varphi_{z_2}(w))|^2)^{{\alpha}}}}\right)-\log\left({1-|\phi^{(2)}(\varphi_{z_2}(w))|^2}\right)		
	\nonumber\\
		\simeq&-\frac{3}{2}\log\left({1-\frac{|\phi^{(1)}(\varphi_{z_2}(w))|^2}{(1-|\phi^{(2)}(\varphi_{z_2}(w))|^2)^{{\alpha}}}}\right)-(1+\frac{\alpha}{2})\log\left({1-|\phi^{(2)}(\varphi_{z_2}(w))|^2}\right)		
	\nonumber\\\simeq&	-\log s(\phi_{f_{\alpha}(z)}\circ\varphi_{z_2}(w),0).
			\end{align}
			Hence  the inequality $-\log s(z,w)\simeq -\log s(\phi_{f_{\alpha}(z)}\circ\varphi_{z_2}(w),0)$ is proved. As a consequence, there exists constants $C_1$, $C_2$, $a$, and $b$ such that $$C_1s^a\geq s(\phi_{f_{\alpha}(z)}\circ\varphi_{z_2}(w),0)\geq C_2s^b.$$ This inequality then implies Property (1) of Lemma 4.2 $$D(0,p(C_1s^a)\subseteq \phi_{f_{\alpha}(z)}\circ\varphi_{z_2} (D(z,p(s)))\subseteq D(0,p(C_2s^b)).$$
			We turn to prove Property (1) of Lemma 4.2 by first showing that $$ \lambda(D(z,p(s)))\lesssim c(s)\lambda(D(0,p(C_2s^b))).$$ For sufficiently small $s$, it is shown from above that there exists a constant $C_2>0$ and $b>0$ such that the set $\phi_{f_{\alpha}(z)}\circ\varphi_{z_2}(D(z,p(s)))\subseteq D(0,p(C_2s^b))$. Therefore 
			\begin{align}
			\lambda(D(z,p(s)))&=\int_{D(z,p(s))}K_{U^{\alpha}}(w;\bar w)d\sigma(w)\nonumber
			\\&=\int_{\phi_{f_{\alpha}(z)}(\varphi_{z_2}(D(z,p(s))))}K_{U^{\alpha}}(\varphi_{z_2}(\phi_{f_{\alpha}(z)}(t));\overline {\varphi_{z_2}(\phi_{f_{\alpha}(z)}(t))})|J(\varphi_{z_2}\circ\phi_{f_{\alpha}(z)}(t)|^2d\sigma(t)\nonumber
			\\&\leq\int_{D(0,p(C_2s^b))}K_{U^{\alpha}}(\varphi_{z_2}(\phi_{f_{\alpha}(z)}(t));\overline {\varphi_{z_2}(\phi_{f_{\alpha}(z)}(t))})|J(\varphi_{z_2}\circ\phi_{f_{\alpha}(z)}(t)|^2d\sigma(t)\nonumber
			\\&\leq\int_{D(0,p(C_2s^b))}K_{U^{\alpha}}(\phi_{f_{\alpha}(z)}(t);\overline {\phi_{f_{\alpha}(z)}(t)})|J(\phi_{f_{\alpha}(z)}(t)|^2d\sigma(t)\nonumber
				\\&\lesssim\int_{D(0,p(C_2s^b))}K_{U^{\alpha}}(t;\bar t)\frac{(1-|\phi^{(2)}(t)|^2)^{1-\alpha}}{(1-|t|^2)^{1-\alpha}}d\sigma(t).
			\end{align}
			We claim that $\frac{(1-|\phi^{(2)}(w)|^2)^{1-\alpha}}{(1-|w|^2)^{1-\alpha}}\simeq c(s)$ for some constant that depends only on $s$. Then we have $\lambda(D(z,p(s)))\lesssim c(s)\lambda(D(0,p(C_2s^b)))$, which completes the proof.
			
			For $t\in D(0,p(C_2s^b))$, we have
			\begin{align}
		-\log s(0,t)\simeq-\log\left(1-\frac{|t_1|^2}{(1-|t_2|^2)^{\alpha-1}}-|t_2|^2\right)\lesssim -\log C_2s^b\simeq-\log s.
			\end{align}
			Hence there exists a constant $c_1(s)>0$, such that $\left(1-\frac{|t_1|^2}{(1-|t_2|^2)^{\alpha-1}}-|t_2|^2\right)>c_1(s)$. Since 
		$1-|t_2|^2>\left(1-\frac{|t_1|^2}{(1-|t_2|^2)^{\alpha-1}}-|t_2|^2\right)$, we also have $1\geq 1-|t_2|^2>c_1(s)$. By (\ref{1.14}),
	\begin{align}
	1\geq1-|\phi^{(2)}(t)|^2&=1-\frac{|t_2|^2(1-|\phi^{(1)}(t)|^{2/\alpha})}{(1-|t_1|^{2/\alpha})}\nonumber
	\\&\geq1-\frac{|t_2|^2}{1-|t_1|^{2/\alpha}}\nonumber
	\\	&\geq 1-|t_2|^2-|t_1|^{2/\alpha}\nonumber
		\\	&=\left(1-\frac{|t_1|^{2/\alpha}}{1-|t_2|^2}\right)(1-|t_2|^2)\nonumber
			\\	&\gtrsim \left(1-\frac{|t_1|^{2}}{(1-|t_2|^2)^{\alpha}}\right)(1-|t_2|^2)\nonumber
				\\	&= \left(1-\frac{|t_1|^2}{(1-|t_2|^2)^{\alpha-1}}-|t_2|^2\right)>c_1(s).
	\end{align}
	Therefore, 	we conclude that \begin{align}\label{4.61}\frac{(1-|\phi^{(2)}(t)|^2)^{1-\alpha}}{(1-|t|^2)^{1-\alpha}}\leq c(s)\end{align} for some constant that depends only on $s$. 
	
	Starting with a constant $C_1$ such that $D(0,p(C_1s^a))\subseteq\phi_{f_{\alpha}(z)}\circ\varphi_{z_2}(D(z,p(s)))$ and then a similar argument  yields the inequality $\lambda (D(z,p(s)))\gtrsim c(s)\lambda(D(0,p(C_1s^a)))$.
	\end{proof}

The following well-known decomposition of a separable metric space in our proof of Proposition 4.1. The proof of this lemma can be found in \cite{Rochberg}. 
\begin{lem}
	Let $(X,d)$ be a separable metric space and $r>0$. For $x\in U^{\alpha}$, let $D(x,d)$ denote the open ball with center $x$ and radius $r>0$ in the metric space $(X,d)$. There is a countable set of points $\{x_j\}$ and a corresponding set of Borel subsets $\{Q_j\}$ of $X$ that satisfy
	\begin{enumerate}
		\item $X=\bigcup_j Q_j;$
		\item $Q_j\bigcap Q_{k}=\emptyset$ for $j\neq k$;
		\item $D(x_j,r)\subset Q_j\subset \{x\in X:d(D(x_j,r),x)\leq r\}$.
	\end{enumerate}
\end{lem}
\begin{proof}[Proof of Proposition 4.1]
We know that $(U^{\alpha},d)$ is a separable metric space.  We choose $s$ as in Lemma 4.2 and then $r=p(s)$ will be close to 1. By Lemma 4.3, there is a collection of points $\{x_j\}\in U^{\alpha}$ and Borel sets $F_j:=Q_j\subset U^{\alpha}$ so that $\mathcal F_r:=\{F_j\}$ is a disjoint covering of $U^{\alpha}$. We first prove that $\operatorname{diam}_dF_j<C(r)<1$. Since $F_j \subseteq\{x\in U^{\alpha}:d(D(x_j,r),x)<r\}$, it suffices to show that
\begin{align}\label{4.7}\operatorname{diam}_d\{x\in U^{\alpha}:d(D(x_j,r),x)<r\}<C(r)<1.\end{align}
Let $z=(z_1,z_2), w=(w_1,w_2)$, and $t=(t_1,t_2)$ be three points in $U^{\alpha}$ such that the distances $d(z,w)=d(w,t)=p(s)=r$. We claim that $d(z,t)<C(s)<1$. Then (\ref{4.7}) holds. Note that the Skwarczy\'nski distance is invariant under the holomorphic automorphism and $\varphi_{z_2}$ sends $(z_1,z_2)$ to $(f_{\alpha}(z),0)$. It is enough to show that $d(z,t)<C(r)$ when $z=(z_1,0)$. For $z=(z_1,0)$,
\begin{align}
-\log s(z,t)&\simeq\log\frac{|1-z_1\bar t_1|^3}{(1-|z_1|^2)^{\frac{3}{2}}\left(1-\frac{|t_1|^2}{(1-|t_2|^2)^\alpha}\right)^{\frac{3}{2}}(1-|t_2|^2)^{\frac{2+\alpha}{2}}}\nonumber
\\&\simeq\log\frac{|1-z_1\bar t_1|^3}{(1-|z_1|^2)^{\frac{3}{2}}\left(1-\frac{|t_1|^2}{(1-|t_2|^2)^\alpha}\right)^{\frac{3}{2}}(1-|t_2|^2)^{\frac{3}{2}}}=I_1.
\end{align} 
Similarly, we have
 \begin{align}
 -\log s(z,w)&\simeq\log\frac{|1-z_1\bar w_1|^3}{(1-|z_1|^2)^{\frac{3}{2}}\left(1-\frac{|w_1|^2}{(1-|w_2|^2)^\alpha}\right)^{\frac{3}{2}}(1-|w_2|^2)^{\frac{3}{2}}}=I_2.
 \end{align} 
 and 
  \begin{align}
 -\log s(t,w)&\simeq\log\frac{|1-t_2\bar w_2|^{2+\alpha}|1-\frac{t_1\bar w_1}{(1-t_2\bar w_2)^{\alpha}}|^3}{(1-|t_2|^2)^{1+\frac{\alpha}{2}}\left(1-\frac{|t_1|^2}{(1-|t_2|^2)^{\alpha}}\right)^{\frac{3}{2}}\left(1-\frac{|w_1|^2}{(1-|w_2|^2)^{\alpha}}\right)^{\frac{3}{2}}(1-|w_2|^2)^{1+\frac{\alpha}{2}}}\nonumber\\
 &\simeq\log\frac{|1-t_2\bar w_2|^{3}|1-\frac{t_1\bar w_1}{(1-t_2\bar w_2)^{\alpha}}|^3}{(1-|t_2|^2)^{\frac{3}{2}}\left(1-\frac{|t_1|^2}{(1-|t_2|^2)^{\alpha}}\right)^{\frac{3}{2}}\left(1-\frac{|w_1|^2}{(1-|w_2|^2)^{\alpha}}\right)^{\frac{3}{2}}(1-|w_2|^2)^{\frac{3}{2}}}
 =I_3.
 \end{align} 
 Since $d(z,w)=d(t,w)=p(s)$, we have $I_2\simeq I_3\simeq-\log s$. Consider $I_1-I_2-I_3$:
 \begin{align}\label{4.11}
 I_1-I_2-I_3&=\log\frac{|1-z_1\bar t_1|^3\left(1-\frac{|w_1|^2}{(1-|w_2|^2)^{\alpha}}\right)^{3}(1-|w_2|^2)^{3}}{|1-z_1\bar w_1|^3|1-t_2\bar w_2|^{3}|1-\frac{t_1\bar w_1}{(1-t_2\bar w_2)^{\alpha}}|^3}.
 \end{align}
 By the triangle inequality of the Bergman distance on the unit disk, we have
 $$\frac{|1-z_1\bar t_1|^3}{|1-z_1\bar w_1|^3}\lesssim \frac{|1-w_1\bar t_1|^3}{(1-|w_1|^2)^3}.$$
Applying this inequality to (\ref{4.11}) yields
  \begin{align}\label{4.131}
 I_1-I_2-I_3&\lesssim\log\frac{|1-w_1\bar t_1|^3\left(1-\frac{|w_1|^2}{(1-|w_2|^2)^{\alpha}}\right)^{3}(1-|w_2|^2)^{3}}{(1-|w_1|^2)^3|1-t_2\bar w_2|^{3}|1-\frac{t_1\bar w_1}{(1-t_2\bar w_2)^{\alpha}}|^3}.
 \end{align}
 Further applying the estimate that $\frac{1-t^2}{1-t^{\frac{2}{\alpha}}}\simeq 1$ to the right hand side of (\ref{4.131}) gives
 \begin{align*}
&\log\frac{|1-w_1\bar t_1|^3\left(1-\frac{|w_1|^2}{(1-|w_2|^2)^{\alpha}}\right)^{3}(1-|w_2|^2)^{3}}{(1-|w_1|^2)^3|1-t_2\bar w_2|^{3}|1-\frac{t_1\bar w_1}{(1-t_2\bar w_2)^{\alpha}}|^3}\\
\lesssim&\log\frac{|1-w_1\bar t_1|^3\left(1-\frac{|w_1|^{\frac{2}{\alpha}}}{1-|w_2|^2}\right)^{3}(1-|w_2|^2)^{3}}{(1-|w_1|^{\frac{2}{\alpha}})^3|1-t_2\bar w_2|^{3}\left(1-\left|\frac{t_1\bar w_1}{(1-t_2\bar w_2)^{\alpha}}\right|\right)^3}\\
\lesssim&\log\frac{|1-w_1\bar t_1|^3\left(1-\frac{|w_2|^{{2}}}{1-|w_1|^{\frac{2}{\alpha}}}\right)^{3}}{\left({|1-t_2\bar w_2|}-{|t_1\bar w_1|^{\frac{1}{\alpha}}}\right)^3}\\
\lesssim&\log\frac{|1-w_1\bar t_1|^3\left(1-\frac{|w_2|^{{2}}}{1-|w_1|^{\frac{2}{\alpha}}}\right)^{3}}{(1-|t_1\bar w_1|)^3\left(1-\frac{|t_2\bar w_2|}{1-|t_1\bar w_1|^{\frac{1}{\alpha}}}\right)^3}\\
\lesssim&\log\frac{|1-w_1\bar t_1|^3}{(1-|t_1\bar w_1|)^3}.
 \end{align*}

When $|t_2\bar w_2|>s^\frac{1}{3}$, we have $1-|t_1\bar w_1|^\frac{1}{\alpha}>|t_2\bar w_2|>s^\frac{1}{3}$. Then 
$$\log\frac{|1-w_1\bar t_1|^3}{(1-|t_1\bar w_1|)^3}\lesssim\log \frac{|1-w_1\bar t_1|^3}{(1-|t_1\bar w_1|^\frac{1}{\alpha})^3}<\log \frac{|1-w_1\bar t_1|^3}{s}\lesssim-\log s.$$
Hence $d(z,t)=\sqrt{1-s(z,t)}<\sqrt{1-s^b}$ for some constant $b$. 

When $|t_2\bar w_2|<s^\frac{1}{3}$, we write $1-t_2\bar w_2=r_1e^{i\theta_1}$ where $r_1=|1-t_2\bar w_2|$. Using trigonometric geometry, we have $$|\theta_1|\simeq|\sin\theta_1|\leq|t_2\bar w_2|<s^\frac{1}{3}.$$ We also write the function $$\frac{t_1\bar w_1}{(1-t_2\bar w_2)^{\alpha}}=r_2e^{i\theta_2}.$$
Thus $t_1\bar w_1=r_1^{\alpha}r_2e^{i(\alpha\theta_1+\theta_2)}$. We claim there is a constant $c>0$ such that 
\begin{align}\label{4.13}
|1-t_1\bar w_1|<s^{-c}(1-|t_1\bar w_1|).
\end{align} Assuming the claim, we have 
$$\log\frac{|1-w_1\bar t_1|^3}{(1-|t_1\bar w_1|)^3}\lesssim\log \frac{s^{-c}(1-|t_1\bar w_1|)}{(1-|t_1\bar w_1|^\frac{1}{\alpha})}\lesssim\log \frac{1}{s^c}\lesssim-c\log s,$$
which implies that $d(z,t)<C(s)<1$, and the proof for Part (2) of Proposition 4.1 is complete. We show the claim by contradiction. Suppose (\ref{4.13}) is not true. Then for any large constant $c$, there exists points $t=(t_1,t_2)$ and $w=(w_1,w_2)$ in $U^{\alpha}$ such that 
$$|1-t_1\bar w_1|\geq s^{-c}(1-|t_1\bar w_1|).$$
Thus we have 
\begin{align}
s^{-2c}\leq&\frac{|1-t_1\bar w_1|^2}{(1-|t_1\bar w_1|)^2}\nonumber\\=&\frac{1+r_1^2r_2^2-2r_1r_2\cos(\alpha\theta_1+\theta_2)}{1+r_1^2r_2^2-2r_1r_2} 
\nonumber\\=&1+\frac{2r_1r_2(1-\cos(\alpha\theta_1+\theta_2))}{1+r_1^2r_2^2-2r_1r_2} 
\nonumber\\=&1+\frac{4r_1r_2\sin^2(\frac{\alpha\theta_1+\theta_2}{2})}{1+r_1^2r_2^2-2r_1r_2} 
\nonumber\\=&1+\frac{4r_1r_2(\sin(\frac{\alpha\theta_1}{2})\cos(\frac{\theta_2}{2})+\cos(\frac{\alpha\theta_1}{2})\sin(\frac{\theta_2}{2}))^2}{1+r_1^2r_2^2-2r_1r_2} 
\nonumber\\\lesssim&\frac{r_2\sin^2(\frac{\alpha\theta_1}{2})}{(1-|t_1\bar w_1|)^2}+\frac{r_2|\sin(\frac{\alpha\theta_1}{2})||\sin(\frac{\theta_2}{2})|}{(1-|t_1\bar w_1|)^2}+\frac{r_2\sin^2(\frac{\theta_2}{2})}{(1-|t_1\bar w_1|)^2}.
\end{align}
Since $|\theta_1|\simeq|\sin\theta_1|\leq|t_2\bar w_2|<s^\frac{1}{3}$ for sufficient small $s$, we have $\sin^2(\frac{\alpha\theta_1}{2})\lesssim \frac{\alpha^2|t_2\bar w_2|^2}{4}$ and hence 
\begin{align}
&\frac{r_2\sin^2(\frac{\alpha\theta_1}{2})}{(1-|t_1\bar w_1|)^2}+\frac{r_2|\sin(\frac{\alpha\theta_1}{2})||\sin(\frac{\theta_2}{2})|}{(1-|t_1\bar w_1|)^2}+\frac{r_2\sin^2(\frac{\theta_2}{2})}{(1-|t_1\bar w_1|)^2}\nonumber\\\lesssim&\frac{r_2|t_2\bar w_2|^2}{(1-|t_1\bar w_1|)^2}+\frac{r_2|t_2\bar w_2||\sin(\frac{\theta_2}{2})|}{(1-|t_1\bar w_1|)^2}+\frac{r_2\sin^2(\frac{\theta_2}{2})}{(1-|t_1\bar w_1|)^2}
\nonumber\\\lesssim&\frac{r_2|\sin(\frac{\theta_2}{2})|}{(1-|t_1\bar w_1|)}+\frac{r_2\sin^2(\frac{\theta_2}{2})}{(1-|t_1\bar w_1|)^2}.
\end{align}
Note that $s(w,t)=|\langle k_w,k_t\rangle|=s$. There exists a constant $b>0$ such that
\begin{align}\frac{|1-\frac{t_1\bar w_1}{(1-t_2\bar w_2)^{\alpha}}|^2}{\left(1-\left|\frac{t_1\bar w_1}{(1-t_2\bar w_2)^{\alpha}}\right|\right)^2}=&\frac{(1-r_2)^2+4r_2\sin^2(\frac{\theta_2}{2})}{(1-r_2)^2}\nonumber\\\leq&\frac{(1-r_2)^2+4r_2\sin^2(\frac{\theta_2}{2})}{\left(1-\frac{|t_1\bar w_1|}{(1-|t_2\bar w_2|)^{\alpha}}\right)^2}\nonumber\\\leq& \frac{\left|1-\frac{t_1\bar w_1}{(1-t_2\bar w_2)^{\alpha}}\right|^2}{\left(1-\frac{|t_1\bar w_1|}{(1-|t_2\bar w_2|)^{\alpha}}\right)^2}\lesssim s^{-b}.\end{align}
Therefore
\begin{align}
&s^{-2c}\lesssim \frac{r_2|\sin(\frac{\theta_2}{2})|}{(1-|t_1\bar w_1|)}+\frac{r_2\sin^2(\frac{\theta_2}{2})}{(1-|t_1\bar w_1|)^2}\lesssim\frac{r_2\sin^2(\frac{\theta_2}{2})}{\left(1-\frac{|t_1\bar w_1|}{(1-|t_2\bar w_2|)^{\alpha}}\right)^2}\lesssim s^{-b}.
\end{align}
Since $s$ is chosen to be sufficiently small,  the inequality above implies that $2c<b$ which contradicts the assumption that $c$ can be arbitrarily large. Hence (\ref{4.7}) is proved.
 
  It remains to show that every point of $U^\alpha$ belongs to at most $N(r)$ of sets $$G_j:=\{z\in U^{\alpha}:d(z,F_j)\leq r\}.$$ Since $F_j \subseteq\{x\in U^{\alpha}:d(D(x_j,r),x)<r\}$, we clearly have $G_j\subset D(x_j,p(s^b))$ for some constant $b>1$.  It suffices to prove that the corresponding balls $D(x_j,p(s^b))$ have a finite intersection property. Our argument about $\operatorname{diam}_d F_j$ also implies that the set $$\{x\in U^{\alpha}:d(D(x_j,p(s^b)),x)<p(s^b)\}$$ is contained in $D(x_j,p(s^{b^2}))$. Suppose that $z\in \cap_{l=1}^ND(x_l,p(s^b))$. Then we have
$$\bigcup_{l=1}^ND(x_l,p(s^b))\subset D(x_k,p(s^{b^5})),$$
where we have fixed one of the $N$ balls, the one centered at $x_k$. Since $D(x_l,p(s))$ are disjoint,
$$\sum_{l=1}^N\lambda(D(x_l,p(s)))\leq\lambda(D(x_k,p(s^{b^{5}}))).$$
Therefore Lemma 4.2 implies that
$$c(s)\lambda (D(0,p(s)))N\simeq\sum_{l=1}^N\lambda(D(x_l,p(s)))\leq\lambda(D(x_k,p(s^{b^{5}})))\simeq c({s^b}^5)\lambda(D(0,p(s^{b^{5}}))),$$
and
$$N\lesssim \frac{c(s^{b^{5}})\lambda(D(0,p(s^{b^{5}})))}{c(s)\lambda (D(0,p(s)))}=:C(s).$$
Here $C(s)$ is a constant depending only on $s$ and the proof is complete. 
\end{proof} 

With the decomposition from Proposition 4.1, we obtain  the following localization property which is a crucial step towards the compactness results.
\begin{prop}Let $T:A^2(U^{\alpha})\rightarrow A^2(U^{\alpha})$ be a linear operator. If 
\begin{align}
&\sup\limits_{z\in U^{\alpha}}\|V_{f_{\alpha}(z)}U_{z}Tk_{z}(w)\|_{L^p(U^{\alpha})}<\infty, \text{ and }\\
&\sup\limits_{z\in U^{\alpha}}\|V_{f_{\alpha}(z)}U_{z}T^*k_{z}(w)\|_{L^p(U^{\alpha})}<\infty.
\end{align}
for some $p>4$, then for every $\epsilon>0$ there exists $r>0$ such that for the covering $\mathcal F_r=\{F_j\}$ from Proposition 4.1
\begin{align}
\|T-\sum_{j}M_{1_{F_j}}TPM_{1_{G_j}}\|<\epsilon.
\end{align}
\end{prop}
\begin{proof}
 We set $R(z,w)=\sum_j1_{F_j}(z)1_{G^c_j}(w)|\langle T^*K_z,K_w\rangle |$, $h(z)=\|K_z\|^t$. Let $\{F_j\}$ and $\{G_j\}$ be as in the proof of Lemma 4.2 and Proposition 4.1 provided the radius $r=p(s)$. Since $\{F_j\}$ form a covering for $U^{\alpha}$ there exists a unique $j$ such that $z\in F_j$.
	Then we have
	\begin{align}
	\int_{U^{\alpha}}R(z,w)\|K_w\|^{t}d\sigma(w)&=\int_{U^{\alpha}}1_{F_j}(z)1_{G^c_j}(w)|\langle T^*K_z,K_w\rangle |\|K_w\|^{t}d\sigma(w)\nonumber\\
	&=\int_{G^c_j}1_{F_j}(z)|\langle T^*K_z,K_w\rangle |\|K_w\|^{t}d\sigma(w)\nonumber\\
		&\leq\int_{D(z,p(s))^c}|\langle T^*K_z,K_w\rangle |\|K_w\|^{t}d\sigma(w).
	\end{align}
	By a same argument as in the proof of Theorem 3.1 and the following fact from Lemma 4.2 $$\phi_{f_{\alpha}(z)}\circ\varphi_{z_2}(D(z,p(s)))\supseteq D(0,p(C_1s^a)),$$ we have 
	\begin{align}
&\int_{D(z,p(s))^c}|\langle T^*K_z,K_w\rangle |\|K_w\|^{t}d\sigma(w)\nonumber\\
\lesssim &\int_{D(0,p(C_1s^a))^c}\frac{\left|  V_{f_\alpha(z)}U_{z}T^*k_z(w)\right|(1-|w_2|^2)^{\frac{\alpha-1}{2}}\left\| K_{\varphi_{z_2}(\phi_{f_{\alpha}(w)}(w))} \right\|^{t-1}\|K_z\|^{1-t}}{(1-|\phi^{(2)}(w)|^2)^{\frac{\alpha-1}{2}}\|K_{w}\|} d\lambda(w),
	\end{align}
	and therefore for $p>4$
		\begin{align}
	&	\int_{U^{\alpha}}R(z,w)\|K_w\|^{t}d\sigma(w)\nonumber\\\lesssim&\left\| K_z\right\|^{t}\left( \int_{U^{\alpha}} \frac{(1-|w_2|^2)^{\frac{q(\alpha-1)}{2}}\left\| K_{\varphi_{z_2}(\phi_{f_{\alpha}(z)}(w))} \right\|^{q(t-1)}\|K_z\|^{q(1-t)}\|K_w\|^q}{(1-|\phi^{(2)}(w)|^2)^{\frac{q(\alpha-1)}{2}}} d\sigma(w)\right)^{\frac{1}{q}}\\
		&\times\left(\int_{D(0,p(C_1s^a))^c} |V_{f_{\alpha}(z)}U_{z}T^*k_z(w)|^pd\sigma(w)\right)^\frac{1}{p}\nonumber\\
		\lesssim&\left\| K_z\right\|^{t}	\left(\int_{D(0,p(C_1s^a))^c} |V_{f_{\alpha}(z)}U_{z}T^*k_z(w)|^pd\sigma(w)\right)^\frac{1}{p}.
		\end{align}
	Since  $\sup\limits_{z\in U^{\alpha}}\|V_{f_{\alpha}(z)}U_{z}Tk_{z}(w)\|_{L^p(U^{\alpha})}<\infty$, the integral $\left(\int_{D(0,p(C_1s^a))^c} |V_{f_{\alpha}(z)}U_{z}T^*k_z(w)|^pd\sigma(w)\right)^\frac{1}{p}$ approaches 0 when $s\rightarrow 0$.
		
		Next, we check the second condition. Fix $w\in U^{\alpha}$. Let $J$ be a subset of all indices $j$ such that $w\notin G_j$. Then $\cup_{j\in J}F_j\subseteq D(w,r)^c$ and hence
		\begin{align}
		\int_{U^{\alpha}}R(z,w)\|K_z\|^t d\sigma(z)&=\int_{\cup_{j\in J}F_j}|\langle T K_w,K_z\rangle |\|K_z\|^t d\sigma(z)\nonumber\\
		&\leq \int_{D(w,r)^c}|\langle T K_w,K_z\rangle |\|K_z\|^t d\sigma(z).
		\end{align}
		Using the same estimates as above, we obtain that 
		\begin{align}
		\int_{U^{\alpha}}R(z,w)\|K_z\|^t d\sigma(z)\lesssim C(r)\|K_z\|^t,
		\end{align}
		where $C(r)\rightarrow 0$ as $r\rightarrow1$. This proves the proposition.
\end{proof}
Proposition 4.4 together with the following lemma gives an approximation of the bounded operator $T$ using a series of compact operators. See \cite{Englis} for the proof of the lemma.

\begin{lem}
Let $G$ be a precompact Borel set $G$ in $U^{\alpha}$.  Then $T_{1_{G}}$ is compact on $A^2(U^{\alpha})$.
\end{lem} 
	\section{The compactness of operators}
Now we are ready to state and prove the main theorem for the compactness.
\begin{thm}
	Let $T:A^2(U^{\alpha})\rightarrow A^2(U^{\alpha})$ be a linear operator. If	\begin{align}\label{5.1}
	&\sup\limits_{z\in U^{\alpha}}\|V_{f_{\alpha}(z)}U_{z}Tk_{z}(w)\|_{L^p(U^{\alpha})}<\infty;\\\label{5.2}
	&\sup\limits_{z\in U^{\alpha}}\|V_{f_{\alpha}(z)}U_{z}T^*k_{z}(w)\|_{L^p(U^{\alpha})}<\infty,
	\end{align}
	for some $p>4$, then $\lim_{d(z,0)\rightarrow1}\|Tk_z\|=0$ if and only if $T$ is compact.
\end{thm}
\begin{proof}
	Suppose $T$ is compact. Then $T$ sends weakly convergent sequences to strongly convergent sequences. To prove that  $\lim_{d(z,0)\rightarrow1}\|Tk_z\|=0$, it suffices to show that the weak limit of $k_z$ is zero as $d(z,0)\rightarrow1$. Since $\{k_w\}$ is dense in $A^2(U^{\alpha})$, it is enough to prove that for each fixed $w\in U^{\alpha}$, $\langle k_z,k_w\rangle\rightarrow 0$ as $d(z,0)\rightarrow 1$. Note that
	\begin{align}
|\langle k_z,k_w\rangle|=\frac{|K_{U^{\alpha}}(z;\bar w)|}{\|K_z\|\|K_w\|}.
	\end{align}
	By Formulas (\ref{KD}) and (\ref{KD1}), both $\|K_w\|$ and $|K_{U^{\alpha}}(z;\bar w)|$ are bounded for fixed $w$. 
	Therefore $\langle k_z,k_w\rangle\rightarrow 0$ as $d(z,0)\rightarrow 1$.
	
We turn to prove the other direction of the statement.  Fix a small $\epsilon>0$. By Proposition 4.4, there exists a large $s$  such that for the covering $\mathcal F_r=\{F_j\}$ associated to $s$
\begin{align}
\|T-\sum_{j}M_{1_{F_j}}TPM_{1_{G_j}}\|<\epsilon.
\end{align}
By Lemma 4.5, the Toeplitz operators $PM_{1_{G_j}}$ are compact. The sum $\sum_{j\leq m}M_{1_{F_j}}TPM_{1_{G_j}}$ is compact for every $m\in\mathbb N$. So, it is enough to show that 
\begin{align}
\limsup\limits_{m\rightarrow\infty}\|\sum_{j>m}M_{1_{F_j}}TPM_{1_{G_j}}\|\lesssim \epsilon.
\end{align}
Let $f\in A^2(U^{\alpha})$ be arbitrary of norm no greater than 1. Then,
\begin{align}
\|T_mf\|^2&=\sum_{j>m}\|M_{1_{F_j}}TPM_{1_{G_j}}f\|^2\nonumber
\\&=\sum_{j>m}\frac{\|M_{1_{F_j}}TPM_{1_{G_j}}f\|^2}{\|M_{1_{G_j}}f\|^2}\|M_{1_{G_j}}f\|^2\leq N(s)\sup\limits_{j>m}\|M_{1_{F_j}}Tl_j\|^2,
\end{align}
where $N(s)$ is the number of overlaps when radius $r=p(s)$ from Proposition 4.1 and
\begin{align}
l_j:=\frac{PM_{1_{G_j}}f}{\|M_{1_{G_j}}f\|}.
\end{align}
Therefore,
\begin{align}
\|T_m\|\leq N(s)\sup_{j>m}\sup_{\|f\|= 1}\left\{\|Tl_j\|:l_j=\frac{PM_{1_{G_j}}f}{\|PM_{1_{G_j}}f\|}\right\},
\end{align}
and hence 
\begin{align}
\limsup\limits_{m\rightarrow\infty}\|T_m\|\leq N(s) \limsup\limits_{j\rightarrow\infty}\sup_{\|f\|= 1}\left\{\|Tl_j\|:g=\frac{PM_{1_{G_j}}f}{\|PM_{1_{G_j}}f\|}\right\}.
\end{align}
Let $\epsilon>0$. There exists a normalized sequence $\{f_j\}$ in $A^2(U^{\alpha})$ such that
\begin{align}
N(s) \limsup\limits_{j\rightarrow\infty}\sup_{\|f\|= 1}\left\{\|Tl_j\|:g=\frac{PM_{1_{G_j}}f}{\|PM_{1_{G_j}}f\|}\right\}-\epsilon\leq N(s)\limsup_{j\rightarrow\infty}\|Tg_j\|,
\end{align}
where 
\begin{align}
g_j:=\frac{PM_{1_{G_j}}f_j}{\|PM_{1_{G_j}}f_j\|}=\frac{\int_{G_j}\langle f_j,k_w\rangle k_w d\lambda(w)}{\left(\int_{G_j}|\langle f_j,k_w\rangle|^2d\lambda(w)\right)^{\frac{1}{2}}}.
\end{align}
For each $j$ pick $z_j=(x_j,y_j)\in G_j=\{z\in U^{\alpha}:d(z,F_j)\leq r\}$. There exists $s>0$ such that $G_j\subset D(z_j,s)$ for all $j$. By a change of variables, we have
\begin{align}
g_j(t)=\int_{\phi_{f_\alpha(z_j)}\circ\varphi_{y_j}(G_j)}a_j(\varphi_{y_j}(\phi_{f_\alpha(z_j)}(w)))k_{\varphi_{y_j}(\phi_{f_\alpha(z_j)}(w))}(t)d\lambda(\varphi_{y_j}(\phi_{f_\alpha(z_j)}(w))),
\end{align}
where $a_j(w)$ is defined to be
\begin{align}
\frac{\langle f_j,k_w\rangle}{\left(\int_{G_j}|\langle f_j,k_w\rangle |^2d\lambda(w)\right)^\frac{1}{2}}.
\end{align}
We claim that $g_j=U_{z_j}V_{f_\alpha(z_j)} h_j$, where
\begin{align}\label{3.23}
h_j(t):=\int_{\phi_{f_\alpha(z_j)}\circ\varphi_{y_j}(G_j)}a_j(\varphi_{y_j}(\phi_{f_\alpha(z_j)}(w)))V_{f_\alpha(z_j)}U_{z_j}k_{\varphi_{y_j}(\phi_{f_\alpha(z_j)}(w))}(t)d\lambda(\varphi_{y_j}(\phi_{f_\alpha(z_j)}(w))).
\end{align}
Notice that both $U_{z_j}$ and $V_{f_\alpha(z_j)}$ are involutions, and $h_j\in L^2(U^{\alpha})$. The claim can be proved by showing that for each $g\in L^2(U^{\alpha})$ we have $\langle g,g_j\rangle=\langle U_{z_j}V_{f_\alpha(z_j)} g,h_j\rangle$. This identity can be obtained by a change of variables and using Fubini's Theorem. Now we consider the integrand in (\ref{3.23}).\begin{align}
&V_{f_\alpha(z_j)}U_{z_j}k_{\varphi_{y_j}(\phi_{f_\alpha(z_j)}(w))}(t)\nonumber\\
=&\frac{K_{U^{\alpha}}(\varphi_{y_j}(\phi_{f_\alpha(z_j)}(t));\overline{\varphi_{y_j}(\phi_{f_\alpha(z_j)}(w))})J\varphi_{y_j}\circ\phi_{f_\alpha(z_j)}(t)}{\|K_{\varphi_{z_2}(\phi_{f_\alpha(z)}(w))}\|}\nonumber\\
=&\frac{K_{U^{\alpha}}(\phi_{f_\alpha(z_j)}(t);\overline{\phi_{f_\alpha(z_j)}(w)})J\phi_{f_\alpha(z_j)}(t)}{\overline {J\varphi_{y_j}(\phi_{f_\alpha(z_j)}(w))}\|K_{\varphi_{z_2}(\phi_{f_\alpha(z)}(w))}\|}
\nonumber\\
=&\frac{K_{U^{\alpha}}(\phi_{f_\alpha(z_j)}(t);\overline{\phi_{f_\alpha(z_j)}(w)})J\phi_{f_\alpha(z_j)}(t)}{\|K_{\phi_{f_\alpha(z)}(w)}\|}\frac{|J\varphi_{y_j}(\phi_{f_\alpha(z_j)}(w))|}{\overline {J\varphi_{y_j}(\phi_{f_\alpha(z_j)}(w))}}\nonumber\\
\simeq&C(w){K_{U^{\alpha}}(\phi_{f_\alpha(z_j)}(t);\overline{\phi_{f_\alpha(z_j)}(w)})J\phi_{f_\alpha(z_j)}(t)\overline{J\phi_{f_\alpha(z_j)}(w)}}\nonumber\\\label{5.16}&\times\frac{|J\phi_{f_\alpha(z_j)}(w)J\varphi_{y_j}(\phi_{f_\alpha(z_j)}(w))|(1-|w_2|^2)^\frac{\alpha-1}{2}}{\overline {J\phi_{f_\alpha(z_j)}(w)J\varphi_{y_j}(\phi_{f_\alpha(z_j)}(w))}\|K_w\|(1-|\phi^{(2)}(w)|^2)^\frac{\alpha-1}{2}},
\end{align}
where $C(w)$ is a bounded continuous function on $U^{\alpha}$.
For the simplicity of notations, we set \begin{align}
&\mathcal U_{z_j}(t;w):=
K_{U^{\alpha}}(\phi_{f_\alpha(z_j)}(t);\overline{\phi_{f_\alpha(z_j)}(w)})J\phi_{f_\alpha(z_j)}(t)\overline{J\phi_{f_\alpha(z_j)}(w)},\\\label{5.15}
&b_j(\varphi_{y_j}(\phi_{f_\alpha(z_j)}(w))):=a_j(\varphi_{y_j}(\phi_{f_\alpha(z_j)}(w)))\frac{C(w)|J\phi_{f_\alpha(z_j)}(w)J\varphi_{y_j}(\phi_{f_\alpha(z_j)}(w))|(1-|w_2|^2)^\frac{\alpha-1}{2}}{\overline {J\phi_{f_\alpha(z_j)}(w)J\varphi_{y_j}(\phi_{f_\alpha(z_j)}(w))}\|K_w\|(1-|\phi^{(2)}(w)|^2)^\frac{\alpha-1}{2}}.
\end{align}
Then (\ref{3.23}) can be written as
\begin{align}
h_j(t)=\int_{\phi_{f_\alpha(z_j)}\circ\varphi_{y_j}(G_j)}b_j(\varphi_{y_j}(\phi_{f_\alpha(z_j)}(w)))\mathcal U_{z_j}(t;w)d\lambda(\varphi_{y_j}(\phi_{f_\alpha(z_j)}(w))).
\end{align}
Recall from the proof of Proposition 4.1 that $G_j\subseteq D(z_j,p(s^c))$ for some constant $c$. Hence Lemma 4.2 implies $\phi_{f_\alpha(z_j)}\circ\varphi_{y_j}(G_j)\subseteq D(0,p(C_2s^b))$ for some constants $C_2$ and $b$.
Lemma 4.2 together with Lemma 2.1 also implies the total variation of each member of the sequence of measures $$\{b_j(\varphi_{y_j}(\phi_{f_\alpha(z_j)}(w)))d\lambda(\varphi_{y_j}(\phi_{f_\alpha(z_j)}(w)))\},$$ as elements in the dual of $C\overline{(D(0,p(C_2s^b)))}$, satisfies \begin{align}\label{5.19}\|b_j(\varphi_{y_j}(\phi_{f_\alpha(z_j)}(w)))d\lambda(\varphi_{y_j}(\phi_{f_\alpha(z_j)}(w)))\|\lesssim c(s)\lambda(D(0,p(C_2s^b))).\end{align}  By its definition, $\mathcal U_{z_j}(t;w)$ can be continuously extended to an open neighborhood of the domain $\overline {U^{\alpha}}\times \overline{D(0,p(C_2s^b))}$. Thus $\mathcal U_{z_j}(t;w)$ is uniformly continuous on $\overline {U^{\alpha}}\times \overline{D(0,p(C_2s^b))}$. For the same $\epsilon$ as above,  there exists finitely many disjoint subsets $H^{(l)}_{z_j}$ in $\overline{D(0,p(C_2s^b))}$ and finitely many points $\{\zeta_l\}$ such that the following statements hold
\begin{enumerate}
	\item $\overline{D(0,p(C_2s^b))}=\cup_l H^{(l)}_{z_j}$.
	\item $\zeta_l\in H^{(l)}_{z_j}$ for each $l$.
	\item $|\mathcal U_{z_j}(t;w)-\mathcal U_{z_j}(t;\zeta_l)|<{\epsilon}{ N(s)^{-1}c(s)^{-1}\lambda(D(0,p(C_2s^b)))^{-1}}$ for any $t\in U^{\alpha}$ and $w\in  H^{(l)}_{z_j}$.
\end{enumerate}
Therefore,
\begin{align}
&\left|h_j(t)-\sum_l\int_{H^{(l)}_{z_j}}\mathcal U_{z_j}(t;\zeta_l)b_j(\varphi_{y_j}(\phi_{f_\alpha(z_j)}(w)))d\lambda(\varphi_{y_j}(\phi_{f_\alpha(z_j)}(w)))\right|\nonumber\\=&\left|\sum_l\int_{H^{(l)}_{z_j}}\left(\mathcal U_{z_j}(t;w)-\mathcal U_{z_j}(t;\zeta_l)\right)b_j(\varphi_{y_j}(\phi_{f_\alpha(z_j)}(w)))d\lambda(\varphi_{y_j}(\phi_{f_\alpha(z_j)}(w)))\right|\leq\frac{\epsilon}{N(s)}.
\end{align}
Since $U_{z_j}V_{f_\alpha(z_j)}$ preserves the Lebesgue $L^2$ norm, we have
\begin{align}
N(s)\|TU_{z_j}V_{f_\alpha(z_j)}h_j\|\lesssim \epsilon+N(s)\|TU_{z_j}V_{f_\alpha(z_j)}\sum_lc_l\mathcal U_{z_j}(t;\zeta_l)\|,
\end{align}
where $c_l=\int_{H^{(l)}_{z_j}}b_j(\varphi_{y_j}(\phi_{f_\alpha(z_j)}(w)))d\lambda(\varphi_{y_j}(\phi_{f_\alpha(z_j)}(w)))$. Note that (\ref{4.61}) and (\ref{5.16}) imply
\begin{align}\label{5.22}
U_{z_j}V_{f_\alpha(z_j)}\mathcal U_{z_j}(t;\zeta_l)\simeq C_1(w)k_{\varphi_{y_j}(\phi_{f_\alpha(z_j)}(\zeta_l)}(t),
\end{align}
where $|C_1(w)|$ is also controlled by $s$.
(\ref{5.19}) and (\ref{5.22}) together with triangle inequality give
\begin{align}
N(s)\|TU_{z_j}V_{f_\alpha(z_j)}h_j\|\lesssim \epsilon+N(s)c(s)\lambda(D(0,p(C_2s^b)))\sup_{\zeta\in D(0,p(C_2s^b))}\|Tk_{\varphi_{y_j}(\phi_{f_\alpha(z_j)}(\zeta))}\|.
\end{align}
Finally we have
\begin{align}
\limsup_{m\rightarrow\infty}\|T_m\|\lesssim& N(s)\limsup_{j\rightarrow\infty}\|Tg_j\|+\epsilon\nonumber\\=& N(s)\lim_{j\rightarrow\infty}\|TU_{z_j}V_{f_\alpha(z_j)}h_j\|+\epsilon\nonumber\\\lesssim&N(s\label{key})c(s)\lambda(D(0,p(C_2s^b)))\limsup_{j\rightarrow\infty}\sup_{\zeta\in D(0,p(C_2s^b))}\|Tk_{\varphi_{y_j}(\phi_{f_\alpha(z_j)}(\zeta))}\|+\epsilon.
\end{align}
Since $d(z_j,0)\rightarrow1$ as $j\rightarrow\infty$ and $d(\cdot,\cdot)$ is invariant under $\phi_{f_\alpha(z)}\circ\varphi_{z_2}$ in the sense of Lemma 4.2, we have $d(\varphi_{y_j}(\phi_{f_\alpha(z_j)}(\zeta)),0)\rightarrow1$ for all $\zeta\in D(0,p(C_2s^b))$.
Then $\lim_{d(z,0)\rightarrow1}\|Tk_z\|=0$ implies that $\limsup_{m\rightarrow\infty}\|T_m\|\lesssim\epsilon,$ which completes the proof.
\end{proof}
By Corollary 3.2, Toeplitz operators with $L^\infty$ symbols satisfy conditions (\ref{5.1}) and (\ref{5.2}) in Theorem 5.1. Hence we also have the following result:
\begin{co}
	Let $T$ is a Toeplitz operator with a $L^\infty$ symbol. Then $T$ is compact if and only if $$\lim_{d(z,0)\rightarrow1}\|Tk_z\|=0.$$
\end{co}

\section{Generalizations and remarks}
In this section, we point out possible generalizations of our results along several directions.
\\
\paragraph{1} 
It is natural to further ask if Corollary 5.2 holds when $T$ is a finite sum of finite products of Toeplitz operators with $L^\infty$ symbols. In the classical cases like the Bergman space on the unit ball, the corresponding conditions (\ref{5.1}) and (\ref{5.2}) for such a $T$ can be obtained using the $L^p$ boundedness of the Bergman projection together with the biholomorphic transformation formula for the Bergman kernel function (off the diagonal). See for instances \cite{AxlerZheng,Sua,Mitkovski2013,Mitkovski20142028}. In our case, the Bergman projection on $U^{\alpha}$ is indeed $L^p$ bounded. See for example \cite{Zhenghui2}. On the other hand, the automorphism we use is not holomorphic. Hence it's not clear to us whether (\ref{5.1}) and (\ref{5.2}) is still true when $T$ is a finite sum of finite products of Toeplitz operators with $L^\infty$ symbols. 
\\
\paragraph{2} In the paper, we focus on the domain $\{(z_1,z_2)\in \mathbb C^2: |z_1|^{2/\alpha}+|z_2|^2<1\}$ with $\alpha>0$ for the simplicity of the argument. Explicit formulas like (\ref{2.1}) are available for the Bergman kernel function on a large family of domains, including the generalized Thullen domain $$\{(z,w)\in\mathbb C^n\times\mathbb C^m:\|z\|^{2p}+\|w\|^2<1\},$$ and  the Fock-Bargmann-Hartogs domain $$\{(z,w)\in\mathbb C^n\times \mathbb C^m:\|z\|<e^{-\|w\|^2}\}.$$ See \cite{1978D'A1,1994D'A2,1999BFS,2013Ya,Zhenghui}. It would be interesting to see if the boundedness and compactness results can be generalized to those domains using similar techniques.
\\
\paragraph{3}  Our results focus on $L^2$ operators. $L^p$ analogues of Theorem 5.1 are obtained for the Bargmann-Fock space \cite{BAUER20121323}, Bergman spaces on the disc and unit ball with classical weights \cite{Mitkovski2013}, and the weighted Bergman spaces on the polydisc \cite{Mitkovski}. It is also possible to generalize our results in the $L^p$ setting for $1<p<\infty$. 
\\
\paragraph{4} Another theme in this subject is to use weaker conditions that involves the Berezin transform to determine the compactness of an operator on the Bergman space. For an operator $T\in A^2(\Omega)$, the Berezin transform of $T$ is defined by $B(S)(z):=\langle Tk_z,k_z\rangle$. It is shown in \cite{Sua} that an operator on $A^2(\mathbb B^n)$ is compact if and only if  $T\in \mathcal T_{L^{\infty}}$ and $B(T)(z)$ vanishes as $\|z\|$ goes to $1$. One of the obstacles of relating $\|Tk_z\|$ to $B(T)(z)$ in our setting is that $V_{f_{\alpha}(z)}U_{z}$ is not an operator on $A^2(U^{\alpha})$: to keep the range of the operator be in $A^2(U^{\alpha})$, the domain of $V_{f_{\alpha}(z)}U_{z}$ will not be $A^2(U^{\alpha})$ and will vary for different $z$. It would be interesting to see if $\lim_{d(z,0)\rightarrow \infty}B(T)(z)=0$ also implies $\lim_{d(z,0)\rightarrow \infty}\|Tk_z\|=0$.
\bibliographystyle{alpha}
\bibliography{2}
\end{document}